\documentclass[a4paper,twoside,11pt,leqno]{article}
\usepackage{expl3}
\usepackage{amssymb,amsmath,amsthm}
\usepackage[all]{xy}
\usepackage[nottoc, notlof, notlot]{tocbibind}

\makeatletter
\def\blfootnote{\gdef\@thefnmark{}\@footnotetext}
\makeatother
\title{An inverse mapping theorem for blow-Nash maps on singular spaces\blfootnote{2010 Mathematics Subject Classification. Primary: 14P99, 14E18; Secondary: 14P05, 14P20, 14B05.}\blfootnote{Keywords: arc-analytic, real singularities, motivic integration, Nash functions.}}
\newcommand\shorttitle{A blow-Nash inverse mapping theorem}
\author{Jean-Baptiste Campesato\footnote{\parbox[t][2em][s]{\textwidth}{Univ. Nice Sophia Antipolis, CNRS,  LJAD, UMR 7351, 06100 Nice, France. \newline E-mail address: \href{mailto:Jean-Baptiste.CAMPESATO@unice.fr}{\tt Jean-Baptiste.CAMPESATO@unice.fr}}}}
\date{May 11, 2015}

\usepackage{ifxetex}
\ifxetex
    \usepackage{polyglossia}
    \setdefaultlanguage[variant=us]{english}
     \usepackage{fontspec}
    \usepackage{unicode-math}
\else
    \usepackage[english]{babel}
    \usepackage{euler}
    \newcommand{\nequiv}{\not\equiv}
    \usepackage{tgpagella}
    \usepackage[T1]{fontenc}
    \usepackage[utf8x]{inputenc}
\fi

\usepackage{hyperref}
\usepackage{color}
\definecolor{darkred}{rgb}{.5,0,0}
\definecolor{darkgreen}{rgb}{0,.5,0}
\definecolor{darkblue}{rgb}{0,0,.5}
\hypersetup{
    pdfencoding=unicode,
    colorlinks=true,
    linkcolor=darkred,
    citecolor=darkgreen,
    urlcolor=darkblue,
    bookmarks=false,
    pdftoolbar=true,
    pdfmenubar=true,
    pdffitwindow=true,
    pdfstartview={FitH},
    pdftitle={\shorttitle},
    pdfauthor={Jean-Baptiste Campesato},
    pdfsubject={},
    pdfcreator={},
    pdfproducer={}
    pdfkeywords={},
    pdfnewwindow=true,
}

\usepackage[inner=3cm,outer=2cm,top=1cm,bottom=1cm,includeheadfoot,headsep=.5cm]{geometry}
\usepackage{fancyhdr}
\renewcommand{\thepage}{\arabic{page}} 
\fancypagestyle{plain}{
\fancyhf{}

}

\def\footindent{2em}
\makeatletter
\renewcommand\@makefntext[1]{\leftskip=\footindent\hskip-\footindent\@makefnmark#1}
\makeatother
\usepackage{perpage}
\MakePerPage{footnote}
\makeatletter
\ifxetex
\newcommand*{\fnsymbolsingle}[1]{\ensuremath{\ifcase#1\or\star\or\dagger\or\ddagger\or\textsection\or\|\or\textparagraph\or\else\@ctrerr\fi}} 
\else
\newcommand*{\fnsymbolsingle}[1]{\ensuremath{\ifcase#1\or\star\or\dagger\or\ddagger\or\mathsection\or\|\or\mathparagraph\or\else\@ctrerr\fi}}
\fi
\makeatother
\usepackage{alphalph}
\newalphalph{\fnsymbolmult}[mult]{\fnsymbolsingle}{}

\makeatletter
\renewcommand{\maketitle}{
  \newpage
  \null
  \thispagestyle{plain}
  
  \begin{center}
    {\LARGE \@title \par}
    \vskip 1.5em
    {\large
      \lineskip .5em
      
        \@author
      \par}
    \vskip 1em
    {\large \@date}
  \end{center}
  \par
  \vskip 1.5em}
\makeatother

\theoremstyle{plain}
\newtheorem{thm}{Theorem}[section]
\newtheorem{prop}[thm]{Proposition}
\newtheorem{cor}[thm]{Corollary}
\newtheorem{lemma}[thm]{Lemma}
\theoremstyle{definition}
\newtheorem{defn}[thm]{Definition}

\newtheorem{eg}[thm]{Example}
\newtheorem{rem}[thm]{Remark}

\newtheorem{notation}[thm]{Notation}
\newtheorem{qu}[thm]{Question}

\usepackage{enumitem}

\usepackage{ifmtarg}
\makeatletter
  \newcommand{\TODO}[1]{\@ifmtarg{#1}{\emph{\textbf{TODO}}~}{\emph{\textbf{TODO:}~#1~}}}
\makeatother

\usepackage{ifthen}
\newcommand{\clos}[2][]{\ensuremath{\overline{#2}\ifthenelse{\equal{#1}{}}{}{^{#1}}}}

\begin{document}
\renewcommand{\Im}{\operatorname{Im}}
\newcommand{\Jac}{\operatorname{Jac}}
\newcommand{\Fitt}{\operatorname{Fitt}}
\renewcommand{\d}{\mathrm{d}}
\newcommand{\ord}{\operatorname{ord}}
\renewcommand{\dim}{\operatorname{dim}}
\newcommand{\Sing}{\operatorname{\Sigma}}
\newcommand{\Reg}{\operatorname{Reg}}
\newcommand{\Spec}{\operatorname{Spec}}
\newcommand{\loc}{\operatorname{loc}}
\newcommand{\supp}{\operatorname{supp}}
\renewcommand{\deg}{\operatorname{deg}}
\renewcommand{\det}{\operatorname{det}}
\renewcommand{\Sing}{\operatorname{Sing}}
\renewcommand{\max}{\operatorname{max}}
\newcommand{\rk}{\operatorname{rk}}
\newcommand{\sing}{\mathrm{sing}}
\newcommand{\pring}[1]{\ensuremath{#1^\bullet}} 
\newcommand{\quotient}[2]{\left.\raisebox{.05em}{$#1$}\middle/\raisebox{-.05em}{$#2$}\right.}

\maketitle

\begin{abstract}

A semialgebraic map $f:X\to Y$ between two real algebraic sets is called \emph{blow-Nash} if it can be made Nash (i.e. semialgebraic and real analytic) by composing with finitely many blowings-up with non-singular centers.

We prove that if a blow-Nash self-homeomorphism $f:X\rightarrow X$ satisfies a lower bound of the Jacobian determinant condition then $f^{-1}$ is also blow-Nash and satisfies the same condition.

The proof relies on motivic integration arguments and on the virtual Poincaré polynomial of McCrory--Parusiński and Fichou. In particular, we need to generalize Denef--Loeser change of variables key lemma to maps that are generically one-to-one and not merely birational.
\end{abstract}

\tableofcontents

\section{Introduction}

\emph{Blow-analytic} maps were introduced by T.-C. Kuo in order to classify real singularities \cite{Kuo79,Kuo80,Kuo85}. A map $f:X\to Y$ between real algebraic sets is called blow-analytic if there exists $\sigma:M\rightarrow X$ a finite sequence of blowings-up with non-singular centers such that $f\circ\sigma$ is analytic. In the same vein a semialgebraic map between real algebraic sets is called blow-Nash if the composition with some finite sequence of blowings-up with non-singular centers is Nash (i.e. semialgebraic and analytic). \emph{Arc-analytic} maps were introduced by K. Kurdyka \cite{Kur88}. A map $f: X\to Y$ between two real algebraic sets is called arc-analytic if every real analytic arc on $X$ is mapped by $f$ to a real analytic arc on $Y$. By a result of E. Bierstone and P. D. Milman \cite{BM90} in response to a question of K. Kurdyka, if $f:X\rightarrow Y$ is semialgebraic (i.e. its graph is semialgebraic) and if $X$ is non-singular then $f$ is arc-analytic if and only if it is blow-Nash. When $X$ is non-singular, the set of points where such a map is analytic is dense \cite[5.2]{Kur88} and thus the Jacobian determinant of $f$ is defined everywhere except on a nowhere dense subset of $X$.

The following Inverse Function Theorem is known for $X$ non-singular \cite{FKP10}: \emph{if the Jacobian determinant of a blow-Nash self-homeomorphism $h: X\to X$ is locally bounded from below by a non-zero constant, on the set it is defined, then $h^{-1}$ is blow-Nash and its Jacobian determinant is also locally bounded from below by a non-zero constant on the set it is defined.}

In this paper, we generalize this theorem for singular algebraic sets.

We first introduce, in subsection \ref{sect:geometric}, the notion of generically arc-analytic maps which are maps $f:X\rightarrow Y$ between real algebraic sets such that there exists a nowhere dense subset $S$ of $X$ with the property that every arc on $X$ not entirely included in $S$ is mapped by $f$ to a real analytic arc on $Y$. When $\dim \operatorname{Sing}(X)\ge1$, we see that this condition is strictly weaker than being arc-analytic, otherwise a continuous generically arc-analytic map is an arc-analytic map. Then we show that the semialgebraic generically arc-analytic maps are exactly the blow-Nash ones.

Given $f:X\rightarrow X$ a blow-Nash self-map on a real algebraic set $X$, we have the following diagram
$$\xymatrix{&M \ar[dl]_\sigma \ar[dr]^{\tilde\sigma} \\ X \ar[rr]_f &&X}$$
with $\sigma$ given by a sequence of blowings-up with non-singular centers and $\tilde\sigma$ a Nash map. \\
We may now give an analogue of the lower bound of the Jacobian determinant condition: we say that $f$ satisfies the \emph{Jacobian hypothesis} if the Jacobian ideal of $\sigma$ is included in the Jacobian ideal of $\tilde\sigma$. This condition doesn't depend on the choice of $\sigma$. \\
We are now able to state the main theorem of this paper: \emph{let $f:X\rightarrow X$ be a semialgebraic self-homeomorphism with $X$ an algebraic subset then $f$ is blow-Nash and satisfies the Jacobian hypothesis if and only if $f^{-1}$ satisfies the same conditions.}

Heuristically, the main idea of the proof consists in comparing the \emph{``motivic volume''} of the set of arcs on $X$ and the \emph{``motivic volume''} of the set of arcs on $X$ coming from arcs on $M$ by $\tilde\sigma$. This allows us to prove that we can uniquely lift by $\tilde\sigma$ an arc not entirely included in some nowhere dense subset of $X$. Thereby, such an arc is mapped to an analytic arc by $f^{-1}$. Thus $f^{-1}$ is generically arc-analytic and so blow-Nash. \\
Therefore, we first define the arc space on an algebraic subset $X$ of $\mathbb R^N$ as the set of germs of analytic arcs on $\mathbb R^N$ which lie in $X$, i.e. $\gamma:(\mathbb R,0)\rightarrow X$ such that $\forall f\in I(X),\,f(\gamma)=0$. For $n\in\mathbb N$, we define the space of $n$-jets on $X$ as the set of $n$-jets $\gamma$ on $\mathbb R^N$ such that $\forall f\in I(X),\,f(\gamma(t))\equiv0\mod t^{n+1}$. The subsection \ref{sect:arcsjets} contains some general properties of these objects and some useful results for the proof of the main theorem.

The additive invariant used in order to apply motivic integration arguments is the virtual Poincaré polynomial which associates to a set of a certain class, denoted $\mathcal{AS}$, a polynomial with integer coefficients. We recall the main properties of the collection $\mathcal{AS}$ in subsection \ref{sect:AS}. The virtual Poincaré polynomial was constructed by C. McCrory, A. Parusiński \cite{MP03} and G. Fichou \cite{Fic05}. The subsection \ref{sect:vPp} contains the main properties of this invariant and motivates its use. \\
In order to compute the above-cited \emph{``motivic volumes''}, we first prove a version of Denef--Loeser key lemma for the motivic change of variables formula which fulfills our requirements and with a weaker hypothesis: we don't assume the map to be birational but only generically one-to-one.

Based on these results, we may finally prove there exists a subset on $X$ such that every analytic arc on $X$ not entirely included in this subset may be uniquely lifted by $\tilde\sigma$. This part relies on real analysis arguments and on the fact that an arc not entirely included in the center of a blowing-up may be lifted by this blowing-up. \\

\noindent\textbf{Acknowledgements.} I am very grateful to my thesis advisor Adam Parusiński for his help and support during the preparation of this work. 

\section{Preliminaries}
\subsection{Constructible sets and maps}\label{sect:AS}
Arc-symmetric sets have been first defined and studied by K. Kurdyka in \cite{Kur88}. A subset of an analytic manifold $M$ is arc-symmetric if all analytic arcs on $M$ meet it at isolated points or are entirely included in it. Semialgebraic arc-symmetric sets are exactly the closed sets of a noetherian topology $\mathcal{AR}$ on $\mathbb R^N$. We work with a slightly different framework defined by A. Parusiński in \cite{Par04} and consider the collection of sets $\mathcal{AS}$ defined as the boolean algebra generated by semialgebraic arc-symmetric subsets of $\mathbb P^n_\mathbb R$. 
The advantages of $\mathcal{AS}$ over $\mathcal{AR}$ are that we get a constructible category and a better control of the behavior at infinity. We refer the reader to \cite{KP07} for a survey. \nocite{MP97,MP07}

\begin{defn}[{\cite[2.4]{Par04}}]
Let $\mathcal C$ be a collection of semialgebraic sets. A map between two $\mathcal C$-sets is a $\mathcal C$-map if its graph is a $\mathcal C$-set. We say that $\mathcal C$ is a constructible category if it satisfies the following axioms:
\begin{enumerate}[label=A\arabic*., ref=A\arabic*, nosep]
\item $\mathcal C$ contains the algebraic sets.
\item $\mathcal C$ is stable by boolean operations $\cap$, $\cup$ and $\setminus$.
\item \begin{enumerate}[label=\alph*., labelindent=0pt, leftmargin=1em, nosep]
\item The inverse image of a $\mathcal C$-set by a $\mathcal C$-map is a $\mathcal C$-set.
\item The image of a $\mathcal C$-set by an \emph{injective} $\mathcal C$-map is a $\mathcal C$-set.
\end{enumerate}
\item \label{item:FunClass} Each locally compact $X\in\mathcal C$ is Euler in codimension 1, i.e. there is a semialgebraic subset $Y\subset X$ with $\dim Y\le\dim X-2$ such that $X\setminus Y$ is Euler\footnote{A locally compact semialgebraic set $X$ is Euler if for every $x\in X$ the Euler-Poincaré characteristic of $X$ at $x$ \\ $\chi(X,X\setminus x)=\sum(-1)^i\dim H_i(X,X\setminus x;\mathbb Z_2)$ is odd.}.
\end{enumerate} 
\end{defn}

\begin{rem}
A locally compact semialgebraic set $X$ is Euler in codimension $1$ if and only if it admits a fundamental class for the homology with coefficient in $\mathbb Z_2$. For instance, this property is crucial in the construction of the virtual Poincaré polynomial in order to use the Poincaré duality.
\end{rem}

Given a constructible category $\mathcal C$, we have a notion of $\mathcal C$-closure.

\begin{thm}[{\cite[2.5]{Par04}}]
Let $\mathcal C$ be a constructible category and let $X\in \mathcal C$ be a locally closed set. Then for any subset $A\subset X$ there is a smallest closed subset of $X$ which belongs to $\mathcal C$ and contains $A$. It is denoted by $\clos[\mathcal C]{A}$. Any other closed subset of $X$ that is in $\mathcal C$ and contains $A$ must contain $\clos[\mathcal C]{A}$.
\end{thm}

\begin{rem}[{\cite[2.7]{Par04}}]
If $A$ is semialgebraic then $\dim\clos[\mathcal C]{A}=\dim A$. In particular, if $A\in\mathcal C$ then $\clos[\mathcal C]{A}=A\cup\clos[\mathcal C]{\clos A\setminus A}$ and hence $\dim\left(\clos[\mathcal C]{A}\setminus A\right)<\dim A$.
\end{rem}

\begin{defn}[{\cite[\S4.2]{Par04}}]
A semialgebraic subset $A\subset\mathbb P_\mathbb R^n$ is an \emph{$\mathcal{AS}$-set} if for every real analytic arc $\gamma:(-1,1)\rightarrow\mathbb P_\mathbb R^n$ such that $\gamma((-1,0))\subset A$ there exists $\varepsilon>0$ such that $\gamma((0,\varepsilon))\subset A$.
\end{defn}

Using the proof of \cite[Theorem 2.5]{Par04}, we get the following proposition.
\begin{prop}
There exists a unique noetherian topology on $\mathbb P_\mathbb R^n$ whose closed sets are exactly the closed $\mathcal{AS}$-subsets. 
\end{prop}

\begin{thm}[\cite{Par04}]\ 
\begin{itemize}[nosep]
\item The algebraically constructible sets form a constructible category denoted by $\mathcal{AC}$.
\item $\mathcal{AS}$ is a constructible category.
\item Every constructible category contains $\mathcal{AC}$ and is contained in $\mathcal{AS}$. This implies that each locally compact set in a constructible category is Euler.
\item $\mathcal{AS}$ is the only constructible category which contains the connected components of compact real algebraic sets.
\end{itemize}
\end{thm}

In what follows, \emph{constructible subset} stands for $\mathcal{AS}$-subset, \emph{constructible map} stands for map with constructible graph and \emph{constructible isomorphism} stands for $\mathcal{AS}$-homeomorphism. \\

In our proof of Lemma \ref{lem:CoV} we need the following result which is, in some sense, a replacement of Chevalley's theorem for Zariski-constructible sets over an algebraically closed field.
\begin{thm}[{\cite[4.3]{Par04}}]\label{thm:oddfibers}
Let $A$ be a semialgebraic subset of a real algebraic subset $X$ of $\mathbb P_\mathbb R^n$. Then $A\in\mathcal{AS}$ if and only if there exist a regular morphism of real algebraic varieties $f:Z\rightarrow X$ and $Z'$ the union of some connected components of $Z$ such that $$x\in A\Leftrightarrow \chi\left(f^{-1}(x)\cap Z'\right)\equiv1\mod2$$$$x\notin A\Leftrightarrow \chi\left(f^{-1}(x)\cap Z'\right)\equiv0\mod2$$ where $\chi$ is the Euler characteristic with compact support. \\
In particular the image of an $\mathcal{AS}$-subset by a regular map whose Euler characteristics with compact support of all the fibers are odd is an $\mathcal{AS}$-subset.
\end{thm}

In this paper, we need to work with $\mathcal{AS}$-sets in order to use the virtual Poincaré polynomial discussed below.

In our settings, the noetherianity of the $\mathcal{AS}$ topology will also allow us to prove a version of J. Denef and F. Loeser key lemma for the motivic change of variables formula with a weaker hypothesis. Indeed, we won't assume that the map is birational but only Nash, proper and generically one-to-one.

\subsection{The virtual Poincaré polynomial}\label{sect:vPp}
C. McCrory and A. Parusiński proved in \cite{MP03} there exists a unique additive invariant of real algebraic varieties which coincides with the Poincaré polynomial for (co)homology with $\mathbb Z_2$ coefficients for compact and non-singular real algebraic varieties. Moreover, this invariant behaves well since its degree is exactly the dimension and the leading coefficient is positive. This virtual Poincaré polynomial has been generalized to $\mathcal{AS}$-subsets by G. Fichou in \cite{Fic05}. Furthermore Nash-equivalent $\mathcal{AS}$-subsets have the same virtual Poincaré polynomial. These proofs use the weak factorization theorem \cite{Wlo03,AKMW} in a way similar of what has been done by Bittner in \cite{Bit04} to give a new description of the Grothendieck ring in terms of blowings-up.

\begin{thm}[\cite{Fic05}]
There is an additive invariant $\beta:\mathcal{AS}\rightarrow\mathbb Z[u]$, called the \emph{virtual Poincaré polynomial}, which associates to an $\mathcal{AS}$-subset a polynomial with integer coefficients $\beta(X)=\sum\beta_i(X)u^i\in\mathbb Z[u]$ and satisfies the following properties:
\begin{itemize}
\item $\beta\left(\displaystyle\bigsqcup_{i=1}^kX_i\right)=\displaystyle\sum_{i=1}^k\beta(X_i)$
\item $\beta(X\times Y)=\beta(X)\beta(Y)$
\item For $X\neq\varnothing$, $\deg\beta(X)=\dim X$ and the leading coefficient of $\beta(X)$ is positive\footnote{$\beta(\varnothing)=0$}.
\item If $X$ is non-singular and compact then $\beta_i(X)=\dim H_i(X,\mathbb Z_2)$.
\item If $X$ and $Y$ are Nash-equivalent then $\beta(X)=\beta(Y)$.
\end{itemize}
\end{thm}

The virtual Poincaré polynomial is a more interesting additive invariant than the Euler characteristic with compact support since it stores more information, like the dimension. 
Notice that it is well known that if we forget the arc-symmetric hypothesis and work with all semialgebraic sets, the Euler characteristic with compact support is the only additive invariant \cite{Qua01}.

\subsection{Geometric settings}\label{sect:geometric}
For the sake of convenience, we recall some basics of Nash geometry and arc-analytic maps before introducing generically arc-analytic maps.

A Nash function on an open semialgebraic subset of $\mathbb R^N$ is an analytic function which satisfies a non-trivial polynomial equation. This notion coincides with $C^\infty$ semialgebraic functions. We can therefore define the notion of Nash submanifold in an obvious way. This notion is powerful since we can use tools from both algebraic and analytic geometries, for example we have a Nash implicit function theorem. For more details on Nash geometry, we refer the reader to \cite{BCR} and \cite{Shi87}.

Arc-analytic maps were first introduced by K. Kurdyka in relation with arc-symmetric sets in \cite{Kur88}. These are maps that send analytic arcs to analytic arcs by composition and hence it is suitable to work with arc-analytic maps between arc-symmetric sets. A semialgebraic map $f:M\rightarrow N$ is blow-Nash if there is a finite sequence of blowings-up with non-singular centers $\sigma:\tilde M\rightarrow M$ such that $f\circ\sigma:\tilde M\rightarrow N$ is Nash. Let $M$ be an analytic manifold and $f:M\rightarrow\mathbb R$ a blow-analytic map, since we can lift an analytic arc by a blowing-up with non-singular center of a non-singular variety, $f$ is clearly arc-analytic. Kurdyka conjectured the converse with an additional semialgebraicity\footnote{The question is still open for the general case: is a map blow-analytic if and only if it is subanalytic and arc-analytic?} hypothesis and E. Bierstone and P. D. Milman brought us the proof in \cite{BM90}. A. Parusiński gave another proof in \cite{Par94}. 
We refer the reader to \cite{KP07} for a survey on arc-symmetric sets and arc-analytic maps.

\begin{defn}
Let $U$ be a semialgebraic open subset of $\mathbb R^N$. Then an analytic function $f:U\rightarrow\mathbb R$ is said to be \emph{Nash} if there are polynomials $a_0,\ldots,a_d$ with $a_d\neq0$ such that $$a_d(x)\left(f(x)\right)^d+\cdots+a_0(x)=0$$
\end{defn}

\begin{thm}[{\cite[Proposition 8.1.8]{BCR}}]
Let $U$ be a semialgebraic open subset of $\mathbb R^N$. Then $f:U\rightarrow\mathbb R$ is a Nash function if and only if $f$ is semialgebraic and of class $C^\infty$.
\end{thm}

\begin{defn}
A \emph{Nash submanifold} of dimension $d$ is a semialgebraic subset $M$ of $\mathbb R^p$ such that every $x\in M$ admits a Nash chart $(U,\varphi)$, i.e. there are $U$ an open semialgebraic neighborhood of $0\in\mathbb R^n$, $V$ an open semialgebraic neighborhood of $x$ in $\mathbb R^p$ and $\varphi:U\rightarrow V$ a Nash-diffeomorphism satisfying $\varphi(0)=x$ and $\varphi((\mathbb R^d\times\{0\})\cap U)=M\cap V$. 
\end{defn}

\begin{rem}
A non-singular algebraic subset $M$ of $\mathbb R^p$ has a natural structure of Nash submanifold given by the Jacobian criterion and the Nash implicit function theorem.
\end{rem}

\begin{defn}[{\cite{Kur88}}]
Let $X$ and $Y$ be arc-symmetric subsets of two analytic manifolds. Then $f:X\rightarrow Y$ is arc-analytic if for all analytic arcs $\gamma:(-\varepsilon,\varepsilon)\rightarrow X$ the composition $f\circ\gamma:(-\varepsilon,\varepsilon)\rightarrow Y$ is again an analytic arc.
\end{defn}

\begin{thm}[{\cite{BM90}}]\label{thm:blowNash}
Let $f$ be a semialgebraic map defined on a non-singular algebraic subset. Then $f$ is arc-analytic if and only if $f$ is blow-Nash.
\end{thm}

\begin{rem}\label{rem:blowNashold}
Let $f:X\rightarrow Y$ be a semialgebraic arc-analytic map between algebraic sets. Then $f$ is blow-Nash even if $X$ is singular. Indeed we may first use a resolution of singularities $\rho:U\rightarrow X$ given by a sequence of blowings-up with non-singular centers \cite{Hir64} and apply Theorem \ref{thm:blowNash} to $f\circ\rho:U\rightarrow Y$.
\end{rem}

\begin{rem}
If $M$ is a non-singular algebraic set and $\rho:\tilde M\rightarrow M$ the blowing-up of $M$ with a non-singular center, it is well known that we can lift an arc on $M$ by $\rho$ to an arc on $\tilde M$. This result is obviously false for a singular algebraic set as shown in the following examples. However, if $X$ is a singular algebraic set and $\rho:\tilde X\rightarrow X$ the blowing-up of $X$ with a non-singular center we can lift an arc on $X$ not entirely included in the center\footnote{Such an arc meets the center only at isolated points since it is algebraic and hence arc-symmetric.} and this lifting is unique. 
\end{rem}

\begin{eg}
Consider the Whitney umbrella $X=V(x^2-zy^2)$ and $\rho:\tilde X\rightarrow X$ the blowing-up along the singular locus $I(X_\sing)=(x,y)$. Then we can't lift by $\rho$ an arc included in the handle $\{x=0,\,y=0,\,z<0\}$ ($\rho$ is not even surjective).
\end{eg}

\begin{eg}
This phenomenon still remains in the pure dimensional case. Let $X=V(x^3-zy^3)$. Then $X$ is of pure dimension $2$ and the blowing-up $\rho:\tilde X\rightarrow X$ along the singular locus $I(X_\sing)=(x,y)$ is surjective. However we can't lift the (germ of) analytic arc $\gamma(t)=(0,0,t)$ to an analytic arc. In the $y$-chart, $\tilde X=\{(X,Y,Z)\in\mathbb R^3,\,X^3=Z\}$ and $\rho(X,Y,Z)=(XY,Y,Z)$. Then the lifting of $\gamma$ should have the form $\tilde\gamma(t)=(t^{\frac{1}{3}},0,t)$.
\end{eg}

\begin{rem}\label{rem:Holder}
A continuous subanalytic map $f:U\rightarrow V$ is locally Hölder, i.e. for each compact subset $K\subset U$, there exist $\alpha>0$ and $C>0$ such that for all $x,y\in K$, $\|f(x)-f(y)\|\le C\|x-y\|^\alpha$. See for instance \cite{Har78}, it's a consequence of \cite[\S9, Inequality III]{Hir73}. See also \cite[Corollary 6.7]{BM88}. Or we can directly use Łojasiewicz inequality \cite[Theorem 6.4]{BM88} with $(x,y)\mapsto|f(x)-f(y)|$ and $(x,y)\mapsto|x-y|$.
\end{rem}

The following result will be useful. 
\begin{prop}\label{prop:Puiseux}
Let $f:X\rightarrow Y$ be a surjective proper subanalytic map (resp. proper semialgebraic map) and $\gamma:[0,\varepsilon)\rightarrow Y$ a real analytic (resp. Nash) arc. Then there exist $m\in\mathbb N_{>0}$ and $\tilde\gamma:[0,\delta)\rightarrow X$ analytic (resp. Nash) with $\delta^m\le\varepsilon$ such that $f\circ\tilde\gamma(t)=\gamma(t^m)$.
\end{prop}
\begin{proof}
The proof is divided into two parts. First we use the properness of $f$ to lift $\gamma$ to an arc on $X$ and then we conclude thanks to Puiseux theorem.

Consider the following diagram
$$\xymatrixcolsep{5pc}\xymatrix{
X \ar[d]_f & \tilde X=X\times_Y[0,\varepsilon) \ar[l]_{\mathrm{pr}_X} \ar[d]^{\tilde f} \\
Y & [0,\varepsilon) \ar[l]^{\gamma}
}$$
Let $X_1=\tilde f^{-1}((0,\varepsilon))$. Since $f$ is proper, $\clos{X_1}\setminus X_1\subset\tilde X$. Let $x_0\in\clos{X_1}\setminus X_1$, then by the curve selection lemma (\cite[Proposition 8.1.13]{BCR} for the semialgebraic case) there exists $\gamma_1:[0,\eta)\rightarrow \tilde X$ analytic (resp. Nash) such that $\gamma_1(0)=x_0$ and $\gamma_1((0,\eta))\subset X_1$.
We have the following diagram
$$\xymatrix{
\tilde X \ar[d]_{\tilde f} & [0,\eta) \ar[l]_{\gamma_1} \ar[dl]^h \\
[0,\varepsilon) &
}$$
Then, $h(0)=0$ and $h((0,\eta))\subset(0,\varepsilon)$. Hence there exists $\alpha\in(0,\eta)$ such that $h:[0,\alpha)\rightarrow[0,\beta)$ is a subanalytic (resp. semialgebraic) homeomorphism.

By Puiseux theorem (\cite[Proposition 8.1.12]{BCR} for the semialgebraic case; see also \cite{Paw84}), there exist $m\in\mathbb N_{>0}$ and $\delta\le\beta^{\frac{1}{m}}$ such that $h^{-1}(t^m)$ is analytic (resp. Nash) for $t\in[0,\delta)$. \\
Finally, $\tilde\gamma:[0,\delta)\rightarrow X$ defined by $\tilde\gamma(t)=\mathrm{pr}_X\gamma_1h^{-1}(t^m)$ satisfies $f\circ\tilde\gamma(t)=\gamma(t^m)$.
\end{proof}

In the singular case we will work with a slightly different framework.
\begin{defn}\label{defn:genAA}
Let $X$ and $Y$ be two algebraic sets. A map $f:X\rightarrow Y$ is said to be \emph{generically arc-analytic in dimension $d=\dim X$} if there exists an algebraic subset $S$ of $X$ with $\dim S<\dim X$ such that for all analytic arc $\gamma:(-\varepsilon,\varepsilon)\rightarrow X$ not entirely included\footnote{$\gamma^{-1}(S)\neq(-\varepsilon,\varepsilon)$} in $S$, $f\circ\gamma:(-\varepsilon,\varepsilon)\rightarrow Y$ is analytic.
\end{defn}

If $X$ is non-singular, these maps are exactly the arc-analytic ones.
\begin{lemma}\label{lem:genAA}
Let $X$ be a non-singular algebraic set of dimension $d$\footnote{We mean that every point of $X$ is non-singular of dimension $d$.} and $Y$ an algebraic set. Let $f:X\rightarrow Y$ be a continuous semialgebraic map. If $f$ is generically arc-analytic in dimension $d$ then $f$ is arc-analytic.
\end{lemma}
\begin{proof}
Let $S$ be as in Definition \ref{defn:genAA}. By the Jacobian criterion and the Nash implicit function theorem we may assume that $S$ is locally a Nash subset of $\mathbb R^d$. Taking the Zariski closure we may moreover assume that $S$ is an algebraic subset of $\mathbb R^d$ since it doesn't change the dimension. Let $\gamma:(-\varepsilon,\varepsilon)\rightarrow\mathbb R^d$ be an analytic arc entirely included in $S$.

As in \cite[Corollaire 2.7]{Kur88-2}, by Puiseux theorem, we may assume that $$\begin{array}{lcl}f(\gamma(t))&=&\displaystyle\sum_{i\ge0}b_it^{\frac i p},\,t\ge0 \\ f(\gamma(t))&=&\displaystyle\sum_{i\ge0}c_i(-t)^{\frac i r},\,t\le0\end{array}$$
By \cite[Corollaire 2.8 \& Corollaire 2.9]{Kur88-2}, two phenomena may prevent $f(\gamma(t))$ from being analytic: either one of these expansions has a non-integer exponent or these expansions don't coincide.

To handle the first case, we assume that one of these expansions, for instance for $t\ge0$, has a non-integer exponent, i.e.
$$f(\gamma(t))=\sum_{i=0}^mb_it^i+bt^\frac{p}{q}+\cdots,\,b\neq0,\,m<\frac{p}{q}<m+1,\,t\ge0$$It follows from Remark \ref{rem:Holder} there exists $N\in\mathbb N$ such that for every analytic arc $\delta$ we have $f(\gamma(t)+t^N\delta(t))\equiv f(\gamma(t))\mod t^{m+1}$. We are going to prove that for $\eta\in\mathbb R^d$ generic, the arc $\tilde\gamma(t)=\gamma(t)+t^N\eta$ is not entirely included in $S$ in order to get a contradiction since $f(\tilde\gamma(t))\equiv f(\gamma(t))\mod t^{m+1}$.

Let $t_0\in(-\varepsilon,\varepsilon)\setminus\{0\}$. Since $\dim S<d$, there is $\tilde\eta\in\mathbb R^d\setminus C_{\gamma(t_0)}S$ where $C_{\gamma(t_0)}S$ is the tangent cone of $S$ at $\gamma(t_0)$. Thus there exists $f\in I(S)$ with $f(\gamma(t_0)+x)=f_m(x)+\cdots+f_{m+r}(x)$ where $\deg f_i=i$ and such that $f_m(\tilde\eta)\neq0$. Then $f\left(\gamma(t_0)+st_0^N\tilde\eta\right)=\left(st_0^N\right)^mf_m(\tilde\eta)+\left(st_0^N\right)^{m+1}g(s,t)$ and hence for $s$ small enough the arc $\gamma(t)+t^Ns\tilde\eta$ isn't entirely included in $S$.

Then we prove that the expansions coincide in a similar way. Assume that the expansions are different, i.e.
$$\begin{array}{lcl}f(\gamma(t))&=&\displaystyle\sum_{i=0}^{m-1}a_it^i+bt^m+\cdots,\,t\ge0 \\ f(\gamma(t))&=&\displaystyle\sum_{i=0}^{m-1}a_it^i+ct^m+\cdots,\,t\le0\end{array}$$ with $b\neq c$. As in the previous case, we may construct an arc $\tilde\gamma$ not entirely included in $S$ such that $f\gamma(t)$ and $f\tilde\gamma(t)$ coincide up to order $m+1$. That leads to a contradiction.
\end{proof}

\begin{rem}
If $\dim\Sing(X)=0$ then a generically arc-analytic map $X\rightarrow Y$ is also arc-analytic since the analytic arcs contained in the singular locus are constant.
\end{rem}

\begin{rem}
The previous proof fails when $X$ isn't assumed to be non-singular. Let $X=V(x^3-zy^3)$ and $S=X_\sing=O_z$. Consider (germ of) analytic arc $\gamma(t)=(0,0,t)$ entirely included in $S$. Given any $N\in\mathbb N$ we can't find $\eta(t)$ such that $\tilde\gamma(t)=\gamma(t)+t^N\eta(t)$ isn't entirely included in $S$. Indeed, if we inject the coordinates of $\tilde\gamma$ in the equation $x^3=zy^3$ we get a contradiction considering the orders of vanishing.
\end{rem}

\begin{rem}
A continuous semialgebraic generically arc-analytic in dimension $d=\dim X$ map $f:X\rightarrow Y$ may not be arc-analytic if $\dim\Sing(X)\ge1$. Indeed, let $X=V(x^3-zy^3)$ and $f:X\rightarrow\mathbb R$ be defined by $f(x,y,z)=\frac{x}{y}$. Then $f(0,0,t)=t^{\frac{1}{3}}$ is not analytic.
\end{rem}

In the non-singular case, by Theorem \ref{thm:blowNash}, the blow-Nash maps are exactly the semialgebraic arc-analytic ones. With the following proposition, we notice that more generally the blow-Nash maps are exactly the semialgebraic generically arc-analytic ones.  
\begin{prop}\label{prop:blowNash}
Let $X$ be an algebraic set of dimension $d$. Let $f:X\rightarrow Y$ be a semialgebraic map which is continuous on $\overline{\Reg_dX}$. Then $f$ is generically arc-analytic in dimension $d$ if and only if it is blow-Nash.
\end{prop}
\begin{proof}
Assume that $f$ is generically arc-analytic. Let $\rho:U\rightarrow X$ be a resolution of singularities given by a sequence of blowings-up with non-singular centers, then $f\circ\rho:U\rightarrow Y$ is semialgebraic and generically arc-analytic with $U$ non-singular. Thus $f\circ\rho$ is arc-analytic by Lemma \ref{lem:genAA}. By Theorem \ref{thm:blowNash}, there exists a sequence of blowings-up with non-singular centers $\eta:M\rightarrow U$ such that $f\circ\rho\circ\eta$ is Nash. Finally $f\circ\sigma$ is Nash where $\sigma=\rho\circ\eta:M\rightarrow X$ is a sequence of blowings-up with non-singular centers.

Assume that $f$ is blow-Nash. Then there is $\sigma:M\rightarrow X$ a sequence of blowings-up with non-singular centers such that $f\circ\sigma:M\rightarrow Y$ is Nash. Let $\gamma$ be an arc on $X$ not entirely included in the singular locus of $X$ and the center of $\sigma$, then there is $\tilde\gamma$ an arc on $M$ such that $\gamma=\sigma(\tilde\gamma)$. Thus $f(\gamma(t))=f\circ\sigma(\tilde\gamma(t))$ is analytic.
\end{proof}

\subsection{Arcs \& jets}\label{sect:arcsjets}
Arc spaces and truncations of arcs were first introduced by J. F. Nash in 1964 \cite{Nas95} and their study has gained new momentum with the works of M. Kontsevich \cite{Kon95}, J. Denef and F. Loeser \cite{DL99} on motivic integration. We can notice that K. Kurdyka \cite{Kur88}, A. Nobile \cite{Nob91}, M. Lejeune-Jalabert \cite{L-J90} \cite{GSLJ96}, M. Hickel \cite{Hic93} and others studied arc space and jet spaces before the advent of motivic integration. Most of these works concern the relationship between the singularities of a variety and its jet spaces. \\

In this section, we define the arc space and the jet spaces of a real algebraic set. We first work with the whole ambient Euclidean space and then use the equations of the algebraic set to define arcs and jets on it. Finally we will give and prove a collection of results concerning these objects. \\

The arc space on $\mathbb R^N$ is defined by $$\mathcal L\left(\mathbb R^N\right)=\left\{\gamma:(\mathbb R,0)\rightarrow\mathbb R^N,\,\gamma\text{ analytic}\right\}$$ and, for $n\in\mathbb N$, the set of $n$-jets on $\mathbb R^N$ is defined by $$\mathcal L_n\left(\mathbb R^N\right)=\quotient{\mathcal L\left(\mathbb R^N\right)}{\sim_n}$$ where $\gamma_1\sim_n\gamma_2$ if and only if $\gamma_1\equiv\gamma_2\mod t^{n+1}$. Obviously, $\mathcal L_n\left(\mathbb R^N\right)\simeq\left(\quotient{\mathbb R\{t\}}{t^{n+1}}\right)^N$. We also consider the truncation maps $\pi_n:\mathcal L\left(\mathbb R^N\right)\rightarrow\mathcal L_n\left(\mathbb R^N\right)$ and $\pi^m_n:\mathcal L_m\left(\mathbb R^N\right)\rightarrow\mathcal L_n\left(\mathbb R^N\right)$, where $m>n$. These maps are clearly surjective. \\

Next, assume that $X\subset\mathbb R^N$ is an algebraic subset. The set of analytic arcs on $X$ is $$\mathcal L(X)=\left\{\gamma\in\mathcal L\left(\mathbb R^N\right),\,\forall f\in I(X),\,f(\gamma(t))=0\right\}$$ and, for $n\in\mathbb N$, the set of $n$-jets on $X$ is $$\mathcal L_n(X)=\left\{\gamma\in\mathcal L_n\left(\mathbb R^N\right),\,\forall f\in I(X),\,f(\gamma(t))\equiv0\mod t^{n+1}\right\}$$
When $X$ is singular, we will see that the truncation maps may not be surjective.

\begin{eg}
Let $X\subset\mathbb R^N$ be an algebraic subset, then $\mathcal L_0(X)\simeq X$ and $\mathcal L_1(X)\simeq T^{\mathrm{Zar}}X=\bigsqcup T^{\mathrm{Zar}}_xX$. Indeed, we just apply Taylor expansion to $f(a+bt)$ where $f\in I(X)$ (or we may directly use that the Zariski tangent space at a point is given by the linear parts of the polynomials $f\in I(X)$ after a translation).
\end{eg}

The following lemma is useful to find examples which are hypersurfaces since the constructions of arc space and jet spaces on an algebraic set are algebraic. See \cite[Theorem 4.5.1]{BCR} for a more general result with another proof. We may find similar results for non-principal ideals in \cite[Proposition 3.3.16, Theorem 4.1.4]{BCR}. See also \cite[\S6]{Lam84}.
\begin{lemma}\label{lem:signcriterion}
Let $f\in\mathbb R[x_1,\ldots,x_N]$ be an irreducible polynomial which changes sign, then $I(V(f))=(f)$.
\end{lemma}
\begin{proof}
The following proof comes from \cite[Lemma 6.14]{Lam84}. After an affine change of coordinates, we may assume that $f(a,b_1)<0<f(a,b_2)$ with $a=(a_1,\ldots,a_{N-1})$. Let $g\in I(V(f))$ and assume that $f\nmid g$ in $\mathbb R[x_1,\ldots,x_N]$. In the PID (and hence UFD) $\mathbb R(x_1,\ldots,x_{N-1})[x_N]$, $f$ is also irreducible and $f\nmid g$ too. In this PID, we may find $\varphi$ and $\gamma$ such that $\varphi f+\gamma g=1$. Let $\varphi=\varphi_0/h$ and $\gamma=\gamma_0/h$ with $0\neq h\in\mathbb R[x_1,\ldots,x_{N-1}]$ and $\varphi_0,\gamma_0\in\mathbb R[x_1,\ldots,x_{N-1}][x_N]$. Then $\varphi_0f+\gamma_0g=h$. Let $V$ be a neighborhood of $a$ in $\mathbb R^{N-1}$ such that $\forall v\in V,\,f(v,b_1)<0<f(v,b_2)$. By the IVT, for all $v\in V$, there is $b_1\le b_v\le b_2$ such that $f(v,b_v)=0$, and so $g(v,b_v)=0$. Then $\forall v\in V,\,h(v)=0$ and hence $h\equiv 0$ which is a contradiction.
\end{proof}

\begin{eg}
Let $X=V\left(y^2-x^3\right)$. Since $y^2-x^3$ is irreducible and changes sign, we have $I(X)=\left(y^2-x^3\right)$ by Lemma \ref{lem:signcriterion}. Hence we get,
\begin{align*}
\mathcal L_1(X)&=\left\{(a_0+a_1t,b_0+b_1t)\in\left(\quotient{\mathbb R\{t\}}{t^{2}}\right)^2,\,(b_0+b_1t)^2-(a_0+a_1t)^3\equiv0\mod t^{2}\right\} \\
&=\left\{(a_0+a_1t,b_0+b_1t)\in\left(\quotient{\mathbb R\{t\}}{t^{2}}\right)^2,\,a_0^3=b_0^2,\,3a_1a_0^2=2b_0b_1\right\}
\end{align*}
\begin{align*}
\mathcal L_2(X)&=\left\{\begin{array}{l}(a_0+a_1t+a_2t^2,b_0+b_1t+b_2t^2)\in\left(\quotient{\mathbb R\{t\}}{t^{3}}\right)^2,\\\quad\quad\quad\quad\quad\quad\quad(b_0+b_1t+b_2t^2)^2-(a_0+a_1t+a_2t^2)^3\equiv0\mod t^{3}\end{array}\right\} \\
&=\left\{(a_0+a_1t+a_2t^2,b_0+b_1t+b_2t^2)\in\left(\quotient{\mathbb R\{t\}}{t^{3}}\right)^2,\,\begin{array}{l}a_0^3=b_0^2,\\ 3a_1a_0^2=2b_0b_1,\\ 3a_0^2a_2+3a_0a_1^2=2b_0b_2+b_1^2\end{array}\right\}
\end{align*}
Then the preimage of $(0,t)\in\mathcal L_1(X)$ by $\pi^2_1$ is obviously empty.
\end{eg}

We therefore take care not to confuse the set $\mathcal L_n(X)$ of $n$-jets on $X$ and the set $\pi_n\left(\mathcal L(X)\right)$ of $n$-jets on $X$ which can be lifted to analytic arcs. Thanks to Hensel's lemma and Artin approximation theorem \cite{Art68}, this phenomenon disappears in the non-singular case. 

\begin{prop}\label{prop:surj}
Let $X$ be an algebraic subset of $\mathbb R^N$. The following are equivalent:
\begin{enumerate}[label=(\roman*), ref=(\roman*), nosep]
\item\label{item:s1} For all $n$, $\pi^{n+1}_n:\mathcal L_{n+1}(X)\rightarrow\mathcal L_n(X)$ is surjective.
\item\label{item:s2} For all $n$, $\pi_n:\mathcal L(X)\rightarrow\mathcal L_n(X)$ is surjective.
\item\label{item:s3} $X$ is non-singular.
\end{enumerate}
\end{prop}
\begin{proof}
\ref{item:s3}$\Rightarrow$\ref{item:s2} is obvious using Hensel's lemma and Artin approximation theorem \cite{Art68}. \\
\ref{item:s2}$\Rightarrow$\ref{item:s1} is obvious since $\pi_n=\pi^{n+1}_n\circ\pi_{n+1}$. \\
\ref{item:s1}$\Rightarrow$\ref{item:s3}: Assume that $0$ is a singular point of $X$. We can find $\gamma=\alpha t\in\mathcal L_1(X)$ which doesn't lie in the tangent cone of $X$ at $0$, i.e. such that $f(\alpha t)\nequiv 0\mod t^{m+1}$ for some $f\in I(X)$ of order $m$. Such a $1$-jet can't be lifted to $\mathcal L_m(X)$. 
\end{proof}

The set $\mathcal L_n(X)$ of $n$-jets on $X\subset\mathbb R^N$ can be seen as a algebraic subset of $\mathbb R^{(n+1)N}$. By a theorem of M. J. Greenberg \cite{Gre66}, given an algebraic subset $X\subset\mathbb R^N$, there exists $c\in\mathbb N_{>0}$ such that for all $n\in\mathbb N$, $\pi_n(\mathcal L(X))=\pi^{cn}_n(\mathcal L_{cn}(X))$. Then if we work over $\mathbb C$ the sets $\pi_n(\mathcal L(X))$ are Zariski-constructible by Chevalley theorem. See for instance \cite{L-J90}\footnote{She uses a generalization of \cite[Theorem 6.1]{Art69} instead of Greenberg theorem.}, \cite{GSLJ96} or \cite{DL99}. \\
In our framework, the following example shows that the $\pi_n(\mathcal L(X))$ may not even be $\mathcal{AS}$.

\begin{eg}
Let $X=V\left(x^2-zy^2\right)$. Then for every $a\in\mathbb R$, $\gamma_a(t)=(0,t^2,at^2)\in\mathcal L_2(X)$. Let $\eta(t)=(bt^3+t^4\eta_1(t),t^2+t^3\eta_2(t),at^2+t^3\eta_3(t))\in\mathcal L(\mathbb R^3)$. Let $f(x,y,z)=x^2-zy^2$, then $f(\eta(t))=(b^2-a)t^6+t^7\tilde\eta(t)$. So if $a<0$, $\gamma_a(t)\notin\pi_2(\mathcal L(X))$. However if $a\ge0$, $\gamma_a(t)=\pi_2\left(\sqrt at^3,t^2,at^2\right)\in\pi_2(\mathcal L(X))$.
\end{eg}

\begin{prop}\label{prop:jetsdim}
Let $X\subset\mathbb R^N$ be an algebraic subset of dimension $d$. Then:
\begin{enumerate}[label=(\roman*), nosep]
\item\label{item:dimliftable} $\dim\left(\pi_n(\mathcal L(X))\right)=(n+1)d$
\item\label{item:dimjets} $\dim\left(\mathcal L_n(X)\right)\ge (n+1)d$
\item\label{item:dimfiberliftable} The fibers of $\tilde\pi^{m}_n={\pi^{m}_n}_{|\pi_{m}(\mathcal L(X))}:\pi_m\left(\mathcal L(X)\right)\rightarrow\pi_n\left(\mathcal L(X)\right)$ are of dimension smaller than or equal to $(m-n)d$ where $m\ge n$.
\item\label{item:dimfiber} A fiber $\left(\pi^{n+1}_n\right)^{-1}(\gamma)$ of $\pi^{n+1}_n:\mathcal L_{n+1}(X)\rightarrow\mathcal L_n(X)$ is either empty or isomorphic to $T^{\mathrm{Zar}}_{\gamma(0)}X$.
\end{enumerate}
If moreover we assume that $X$ is non-singular, we get the following statement since $\mathcal L_n(X)=\pi_n\left(\mathcal L(X)\right)$:
\begin{enumerate}[resume*]
\item $\dim\left(\mathcal L_n(X)\right)=(n+1)d$
\end{enumerate}
\end{prop}
\begin{proof}
We first notice that \ref{item:dimliftable} is a direct consequence of \ref{item:dimfiberliftable}.

\ref{item:dimjets} $(\pi^n_0)^{-1}(X\setminus X_\sing)$ is of dimension $(n+1)d$ since the fiber of $\pi^n_0$ over a non-singular point is of dimension $nd$.

\ref{item:dimfiberliftable} We may assume that $m=n+1$. Let $\gamma\in\pi_n(\mathcal L(X))$. We may assume that $\gamma\in(\mathbb R_n[t])^N$.
We consider the following diagram
$$\xymatrix{
&\mathbb R^N\times\mathbb R \ar[dl]_{p_1} \ar[dr]^{p_2}& \\
\mathbb R^N & & \mathbb R 
}$$
with $p_1(x,t)=\gamma(t)+t^{n+1}x$ and $p_2(x,t)=t$. 
Let $\mathfrak{X}=\clos[Zar]{p_1^{-1}(X)\cap\{t\neq0\}}$. For $c\neq0$, $\mathfrak X\cap p_2^{-1}(c)\simeq X$ and $\dim\mathfrak X\cap p_2^{-1}(c)=\dim\mathfrak X-1$. Hence $\dim\mathfrak X\cap p_2^{-1}(0)\le\dim\mathfrak X-1=\dim X$. \\
We are looking for objects of the form $\pi_{n+1}(\gamma(t)+t^{n+1}\alpha(t))$ with $\gamma(t)+t^{n+1}\alpha(t)\in\mathcal L(X)$. Such an $\alpha$ is equivalent to a section of ${p_2}_{|\mathfrak X}$ i.e. $\begin{array}{rcl}\mathbb R & \rightarrow & \mathfrak X \\ t&\mapsto&(\alpha(t),t)\end{array}$. Since we want an arc modulo $t^{n+2}$, we are looking for the constant term of $\alpha$, therefore $(\tilde\pi^{n+1}_n)^{-1}(\gamma)\subset\mathfrak X\cap p_2^{-1}(0)$.

\ref{item:dimfiber} Let $\gamma\in\mathcal L_n(X)$. Let $\eta\in\mathbb R^N$. Assume that $I(X)=(f_1,\ldots,f_r)$. By Taylor expansion we get $$f_i(\gamma+t^{n+1}\eta)\equiv f_i(\gamma(t))+t^{n+1}\left(\nabla_{\gamma(t)}f_i\right)(\eta) \mod t^{n+2}$$
Assume that $f_i(\gamma(t))\equiv t^{n+1}\alpha_i\mod t^{n+2}$. Since $t^{n+1}\left(\nabla_{\gamma(t)}f_i\right)(\eta)\equiv t^{n+1}\left(\nabla_{\gamma(0)}f_i\right)(\eta)\mod t^{n+2}$, we have $$f_i(\gamma+t^{n+1}\eta)\equiv t^{n+1}\left(\alpha_i+\left(\nabla_{\gamma(0)}f_i\right)(\eta)\right) \mod t^{n+2}$$
Hence, $\gamma(t)+t^{n+1}\eta$ is in the fiber $\left(\pi^{n+1}_n\right)^{-1}(\gamma)$ if and only if $\alpha_i+\left(\nabla_{\gamma(0)}f_i\right)(\eta)=0,\,i=1,\ldots,r$.
\end{proof}

An arc-analytic map $f:X\rightarrow Y$ induces a map $f_*:\mathcal L(X)\rightarrow\mathcal L(Y)$. Moreover, if $f:X\rightarrow Y$ is analytic, then we also have maps at the level of $n$-jets $f_{*n}:\mathcal L_n(X)\rightarrow\mathcal L_n(Y)$ such that the following diagram commutes $$\xymatrix{\mathcal L(X) \ar[r]^{f_*} \ar[d]_{\pi_n} & \mathcal L(Y) \ar[d]^{\pi_n} \\ \mathcal L_n(X) \ar[r]_{f_{*n}} & \mathcal L_n(Y)}$$
In particular, if $X$ is non-singular, $\Im f_{*n}\subset\pi_n\left(\mathcal L(Y)\right)$ since $\pi_n:\mathcal L(X)\rightarrow\mathcal L_n(X)$ is surjective. \\

For $M$ a non-singular algebraic set and $\sigma:M\rightarrow X\subset\mathbb R^N$ analytic, we define $\Jac_\sigma(x)$ the Jacobian matrix of $\sigma$ at $x$ with respect to a coordinate system at $x$ in $M$. For $\gamma$ an arc on $M$ with origin $\gamma(0)=x$, we define the order of vanishing of $\gamma$ along $\Jac_\sigma$ by $\ord_t\Jac_\sigma(\gamma(t))=\min\{\ord_t\delta(\gamma(t)),\,\forall\delta\text{ $m$-minor of $\Jac_\sigma$}\}$ where $m=\min(d,N)$ and $\gamma$ is expressed in the local coordinate system. This order of vanishing is independent of the choice of the coordinate system.

The critical locus of $\sigma$ is $C_\sigma=\{x\in M,\,\delta(x)=0,\,\forall\delta\text{ $m$-minor of $\Jac_\sigma$}\}$. If $E\subset M$ is locally described by an equation $f=0$ around $x$ and if $\gamma$ is an arc with origin $\gamma(0)=x$ then $\ord_{\gamma}E=\ord_tf(\gamma(t))$. \\

\section{The main theorem}
\begin{lemma}\label{lem:hyp}
Let $X$ be an algebraic subset of $\mathbb R^N$ and $f:X\rightarrow X$ a blow-Nash map. Let $\sigma:M\rightarrow X$ be a sequence of blowings-up with non-singular centers such that $\tilde\sigma=f\circ\sigma:M\rightarrow X$ is Nash.
$$\xymatrix{&M \ar[ld]_\sigma \ar[rd]^{\tilde\sigma}&\\X \ar[rr]_f&&X}$$
After adding more blowings-up, we may assume that the critical loci of $\sigma$ and $\tilde\sigma$ are simultaneously normal crossing and denote them by $\sum_{i\in I}\nu_iE_i$ and $\sum_{i\in I}\tilde\nu_iE_i$. \\
Then the property \begin{equation}\forall i\in I,\,\nu_i\ge\tilde\nu_i\label{eq:hyp}\end{equation} doesn't depend on the choice of $\sigma$. 
\end{lemma}
\begin{proof}
Given $\sigma_1$ and $\sigma_2$ as in the statement and using Hironaka flattening theorem lemma \cite{Hir75} (which works as it is in the real algebraic case), there exist $\pi_1$ and $\pi_2$ regular such that the following diagram commutes:
$$\xymatrix{
&\widetilde{M} \ar@{-->}[ld]_{\pi_1} \ar@{-->}[rd]^{\pi_2} & & \\
M_1 \ar@/_/[ddr]_{\tilde\sigma_1} \ar[dr]^{\sigma_1} & & M_2 \ar@/^/[ddl]^{\tilde\sigma_2} \ar[dl]_{\sigma_2} \\
& X \ar[d]_f & \\
& X &
}$$
The relation \ref{eq:hyp} means exactly that the Jacobian ideal of $\sigma_i$ is included in the Jacobian ideal of $\tilde\sigma_i$. By the chain rule, the relations at the level $M_i$ are preserved in $\widetilde M$. Again by the chain rule and since the previous diagram commutes, the relations in $M_1$ and $M_2$ must coincide.
\end{proof}

\begin{defn}
We say that a map $f:X\rightarrow X$ as in Lemma \ref{lem:hyp} verifying the relation \eqref{eq:hyp} satisfies the \emph{Jacobian hypothesis}.
\end{defn}

\begin{qu}
May we find a geometric interpretation of this hypothesis?
\end{qu}

The following example is a direct consequence of the chain rule.
\begin{eg}\label{eg:jachyp}
Let $X$ be a non-singular algebraic set and $f:X\rightarrow X$ a regular map satisfying $|\det\d f|>c$ for a constant $c>0$, then $f$ satisfies the \emph{Jacobian hypothesis}.
\end{eg}

\begin{thm}[Main theorem]\label{thm:Main}
Let $X$ be an algebraic subset of $\mathbb R^N$ and $f:X\rightarrow X$ a semialgebraic homeomorphism (for the Euclidean topology). If $f$ is blow-Nash and satisfies the Jacobian hypothesis then $f^{-1}$ is blow-Nash and satisfies the Jacobian hypothesis too.
\end{thm}

By Lemma \ref{lem:genAA} and Proposition \ref{prop:blowNash}, if $X$ is a non-singular algebraic subset we get the following corollary.

\begin{cor}[\cite{FKP10}]
Let $X$ be a non-singular algebraic subset and $f:X\rightarrow X$ a semialgebraic homeomorphism (for the Euclidean topology). If $f$ is arc-analytic and if there exists $c>0$ satisfying $|\det\d f|>c$ then $f^{-1}$ is arc-analytic and there exists $\tilde c>0$ satisfying $|\det\d f^{-1}|>\tilde c$.
\end{cor}

\begin{rem}
We recover \cite[Theorem 1.1]{FKP10} using the last corollary and \cite[Corollary 2.2 \& Corollary 2.3]{FKP10}.
\end{rem}

\section{Proof of the main theorem}

\subsection{Change of variables}
An algebraic version of the following lemma was already known in \cite{Elk73}, \cite{Pop84} or \cite[\S2]{Pop85} with a proof in \cite[4.1]{Swa98}. 
The statement given below is more geometric and the proof is quite elementary. 
\begin{lemma}\label{lem:Hideal}
Let $X$ be a $d$-dimensional algebraic subset of $\mathbb R^N$. We consider the following ideal of $\mathbb R[x_1,\ldots,x_N]$
$$H=\sum_{f_1,\ldots,f_{N-d}\in I(X)}\Delta(f_1,\ldots,f_{N-d})\left((f_1,\ldots,f_{N-d}):I(X)\right)$$
where $\Delta(f_1,\ldots,f_{N-d})$ is the ideal generated by the $(N-d)$-minors of the matrix $\left(\frac{\partial f_i}{\partial x_j}\right)_{{i=1,\ldots,N-d\atop j=1,\ldots,N}}$.
Then $V(H)$ is the singular locus\footnote{By singular locus we mean the complement of the set of non-singular points in dimension $d$ as in \cite[3.3.13]{BCR} (and not the complement of non-singular points in every dimension). We may avoid this precision with the supplementary hypothesis that every irreducible component of $X$ is of dimension $d$ or in the pure dimensional case.} $X_\sing$ of $X$.
\end{lemma}
\begin{proof}
Let $x\notin V(H)$ then there exist $f_1,\ldots,f_{N-d}\in I(X)$, $\delta$ a $(N-d)$-minor of $\left(\frac{\partial f_i}{\partial x_j}\right)_{{i=1,\ldots,N-d\atop j=1,\ldots,N}}$ and $h\in\mathbb R[x_1,\ldots,x_N]$ with $hI(X)\subset(f_1,\ldots,f_{N-d})$ and $h\delta(x)\neq0$. Since $\delta(x)\neq0$, $x$ is a non-singular point of $V(f_1,\ldots,f_{N-d})$. Furthermore we have $X=V(I(X)))\subset V(f_1,\ldots,f_{N-d})\subset V(hI(X))$ and, since $h(x)\neq 0$, in an open neighborhood $U$ of $x$ in $\mathbb R^N$ we have $V(hI(X))\cap U=X\cap U$. Hence $V(f_1,\ldots,f_{N-d})\cap U=X\cap U$. So $x$ is a non-singular point of $X$ by \cite[Proposition 3.3.10]{BCR}. We proved that $X_\sing\subset V(H)$.

Now, assume that $x\in X\setminus X_\sing$. With the notation of \cite[\S3]{BCR}, the local ring $\mathcal R_{X,x}=\quotient{\mathcal R_{\mathbb R^N,x}}{I(X)\mathcal R_{\mathbb R^N,x}}$ is regular, so we may find a regular system of parameters $(f_1,\ldots,f_N)$ of $\mathcal R_{X,x}$ such that $I(X)\mathcal R_{\mathbb R^N,x}=(f_1,\ldots,f_{N-d})\mathcal R_{\mathbb R^N,x}$ by \cite[VI.1.8\&VI.1.10]{Kun85}\footnote{Since $\mathcal R_{\mathbb R^N,x}=\mathbb R[x_1,\ldots,x_N]_{\mathfrak m_x}$} (see also \cite[Proposition 3.3.7]{BCR}). Moreover, we may assume that the $f_1,\ldots,f_{N-d}$ are polynomials. We may use the following classical argument. $\theta:\mathbb R[x_1,\ldots,x_N]\rightarrow\mathbb R^N$ defined by $f\mapsto f(x)$ induces an isomorphism $\theta':\mathfrak m_x/\mathfrak m_x^2\rightarrow\mathbb R^N$. Then $\rk\left(\frac{\partial f_i}{\partial x_j}(x)\right)=\dim\theta((f_1,\ldots,f_{N-d}))$ which is, by $\theta'$, the dimension of $\quotient{((f_1,\ldots,f_{N-d})+\mathfrak m_x^2)}{\mathfrak m_x^2}$ as a subspace of $\quotient{\mathfrak m_x}{\mathfrak m_x^2}$. If we denote by $\mathfrak m$ the maximal ideal of $\mathcal R_{X,x}=\left(\quotient{\mathbb R[x_1,\ldots,x_N]}{(f_1,\ldots,f_{N-d})}\right)_{\mathfrak m_x}$, we have $\quotient{\mathfrak m}{\mathfrak m^2}\simeq\quotient{\mathfrak m_x}{((f_1,\ldots,f_{N-d})+\mathfrak m_x^2)}$. So we have $\dim\left(\quotient{\mathfrak m}{\mathfrak m^2}\right)+\rk\left(\frac{\partial f_i}{\partial x_j}(x)\right)=N$. Furthermore, since $\mathcal R_{X,x}$ is a $d$-dimensional regular local ring, $\dim\left(\quotient{\mathfrak m}{\mathfrak m^2}\right)=d$. Hence $\left(\frac{\partial f_i}{\partial x_j}(x)\right)_{{i=1,\ldots,N-d\atop j=1,\ldots,N}}$ is of rank $N-d$ and so there exists $\delta$ a $(N-d)$-minor of $\left(\frac{\partial f_i}{\partial x_j}\right)_{{i=1,\ldots,N-d\atop j=1,\ldots,N}}$ such that $\delta(x)\neq0$. Assume that $I(X)=(g_1,\ldots,g_r)$ in $\mathbb R[x_1,\ldots,x_N]$. Then $g_i=\sum\frac{f_j}{q_j}$ with $q_j(x)\neq 0$, so $g_ih_i\subset(f_1,\ldots,f_{N-d})$ with $h_i=\prod q_j$. Then $h=\prod h_i$ satisfies $h(x)\neq0$ and $hI(X)\subset (f_1,\ldots,f_{N-d})$. So $x\notin V(H)$. Hence $V(H)\subset X_\sing\cup(\mathbb R^N\setminus X)$.

To complete the proof, it remains to prove that $V(H)\subset X$. Let $x\notin X$. There exist $f_1,\ldots,f_{N-d}\in I(X)$ such that $f_i(x)\neq 0$. We construct by induction $N-d$ polynomials of the form $g_i=a_if_i$ with $g_i(x)\neq 0$ and $(\d g_1\wedge\cdots\wedge\d g_{N-d})_x\neq0$. Suppose that $g_1,\ldots,g_{j-1}$ are constructed, if $(\d g_1\wedge\cdots\wedge\d g_{j-1}\wedge\d f_j)_x\neq 0$, we can take $a_j=1$, so we may assume that $(\d g_1\wedge\cdots\wedge\d g_{j-1}\wedge\d f_j)_x=0$. Then we just have to take some $a_j$ satisfying $(\d g_1\wedge\cdots\wedge\d g_{j-1}\wedge\d a_j)_x\neq 0$ and $a_j(x)\neq0$ since  $(\d g_1\wedge\cdots\wedge\d g_{j-1}\wedge\d (a_jf_j))_x=f_j(x)(\d g_1\wedge\cdots\wedge\d g_{j-1}\wedge\d a_j)_x$. Then we have $g_1,\ldots,g_{N-d}\in I(X)$ whose a $(N-d)$-minor $\delta$ satisfies $\delta(x)\neq0$. Moreover we have $g_i(x)\neq0$ and $g_iI\subset(g_1,\ldots,g_{N-d})$. So $x\notin V(H)$. 
\end{proof}

\begin{defn}
Let $X$ be an algebraic subset of $\mathbb R^N$. For $e\in\mathbb N$, we set $$\mathcal L^{(e)}(X)=\left\{\gamma\in\mathcal L(X),\,\exists g\in H,\,g(\gamma(t))\nequiv0\mod t^{e+1}\right\}$$ where $H$ is defined in Lemma \ref{lem:Hideal}.
\end{defn}

\begin{rem}
$\displaystyle\mathcal L(X)=\left(\bigcup_{e\in\mathbb N}\mathcal L^{(e)}(X)\right)\bigsqcup\mathcal L(X_\sing)$
\end{rem}

\begin{rem}
In \cite{DL99}, Denef--Loeser set $\mathcal L^{(e)}(X)=\mathcal L(X)\setminus\pi^{-1}_e\left(\mathcal L_e(X_\sing)\right)$ and used the Nullstellensatz to get that $I(X_\sing)^c\subset H$ for some $c$ since $X_\sing=V(H)$. Since we can't do that in our case, we defined differently $\mathcal L^{(e)}(X)$.
\end{rem}

The following lemma is an adaptation of Denef--Loeser key lemma \cite[Lemma 3.4]{DL99} to fulfill our settings. 
The aim of the above-mentioned lemma is to allow the proof of a generalization of Kontsevich's birational transformation rule (change of variables) of \cite{Kon95} to handle singularities. We can find a first adaption to our settings in the non-singular case in \cite[Lemma 4.2]{KP03}.

\begin{lemma}\label{lem:CoV}
Let $\sigma:M\rightarrow X$ be a proper generically\footnote{i.e. $\sigma$ is a Nash map which is one-to-one away from a subset $S$ of $X$ with $\dim S<\dim X$.} one-to-one Nash map where $M$ is a non-singular algebraic subset of $\mathbb R^p$ of 
dimension $d$ and $X$ an algebraic subset of $\mathbb R^N$ of dimension $d$. 
For $e,e'\in\mathbb N$, we set $$\Delta_{e,e'}=\left\{\gamma\in\mathcal L(M),\,\ord_t\left(\Jac_\sigma(\gamma(t))\right)=e,\,\sigma_*(\gamma)\in\mathcal L^{(e')}(X)\right\}$$
For $n\in\mathbb N$, let $\Delta_{e,e',n}$ be the image of $\Delta_{e,e'}$ by $\pi_n$.
Let $e,e',n\in\mathbb N$ with $n\ge\operatorname{max}(2e,e')$, then:
\begin{enumerate}[label=(\roman*), ref=\ref{lem:CoV}.(\roman*)]
\item\label{item:CoV0} Given $\gamma\in\Delta_{e,e'}$ and $\delta\in\mathcal L(X)$ with $\sigma_*(\gamma)\equiv\delta\mod t^{n+1}$ there exists a unique $\eta\in\mathcal L(M)$ such that $\sigma_*(\eta)=\delta$ and $\eta\equiv\gamma\mod t^{n-e+1}$.
\item\label{item:CoV1} Let $\gamma,\eta\in\mathcal L(M)$. If $\gamma\in\Delta_{e,e'}$ and $\sigma(\gamma)\equiv\sigma(\eta)\mod t^{n+1}$ then $\gamma\equiv\eta\mod t^{n-e+1}$ and $\eta\in\Delta_{e,e'}$.
\item\label{item:CoV1prime} The set $\Delta_{e,e',n}$ is a union of fibers of $\sigma_{*n}$.
\item\label{item:CoV2} $\sigma_{*n}(\Delta_{e,e',n})$ is constructible and $\sigma_{*n|\Delta_{e,e',n}}:\Delta_{e,e',n}\rightarrow \sigma_{*n}(\Delta_{e,e',n})$ is a piecewise trivial fibration\footnote{By a trivial piecewise fibration, we mean there exist a finite partition of $\sigma_{*n}(\Delta_{e,e',n})$ with constructible parts and a trivial fibration given by a constructible isomorphism over each part.} with fiber $\mathbb R^e$.
\end{enumerate}
\end{lemma}

\begin{rem}
It is natural to use Taylor expansion to prove some approximation theorems concerning power series as we are going to do for \ref{item:CoV0}. For instance, we may find similar argument in \cite{Gre66}, \cite{Art69}, or \cite{Elk73}. For \ref{item:CoV0}, we will follow the proof of \cite[Lemma 3.4]{DL99} with some differences to match our framework. Concerning \ref{item:CoV2}, we can't use anymore the section argument of \cite{DL99} since $\sigma$ is not assumed to be birational.
\end{rem}

\begin{lemma}[Reduction to complete intersection]\label{lem:CI}
Let $X$ be an algebraic subset of $\mathbb R^N$ of dimension $d$. 
For each $e\in\mathbb N$, $\mathcal L^{(e)}(X)$ is covered by a finite number of sets of the form $$A_{h,\delta}=\left\{\gamma\in\mathcal L(\mathbb R^N),\,(h\delta)(\gamma)\nequiv0\mod t^{e+1}\right\}$$ with $\delta$ a $N-d$-minor of the matrix $\left(\frac{\partial f_i}{\partial x_j}\right)_{{i=1,\ldots,N-d\atop j=1,\ldots,N}}$ and $h\in\left((f_1,\ldots,f_{N-d}):I(X)\right)$ for some $f_1,\ldots,f_{N-d}\in I(X)$. \\
Moreover, $$\mathcal L(X)\cap A_{h,\delta}=\left\{\gamma\in\mathcal L\left(\mathbb R^N\right),\,f_1(\gamma)=\cdots=f_{N-d}(\gamma)=0,\,h\delta(\gamma)\nequiv0\mod t^{e+1}\right\}$$
\end{lemma}
\begin{rem}
We may have different polynomials $f_1,\ldots,f_{N-d}$ for two different $A_{h,\delta}$.
\end{rem}

\begin{proof}
By noetherianity, we may assume that $H=(h_1\delta_1,\ldots,h_r\delta_r)$ with $h_i,\delta_i$ as desired. Therefore, $\mathcal L^{(e)}(X)\subset\cup A_{h_i,\delta_i}$. \\
Finally,
\begin{align*}
\mathcal L(X)\cap A_{h,\delta} &= \left\{\gamma\in\mathcal L\left(\mathbb R^N\right),\,\forall f\in I(X),\,f(\gamma)=0,\,h\delta(\gamma)\nequiv0\mod t^{e+1}\right\} \\
&= \left\{\gamma\in\mathcal L\left(\mathbb R^N\right),\,f_1(\gamma)=\cdots=f_{N-d}(\gamma)=0,\,h\delta(\gamma)\nequiv0\mod t^{e+1}\right\}
\end{align*}
Indeed, for the second equality, if $f\in I(X)$ then $hf\in(f_1,\ldots,f_{N-d})$, hence if $\gamma$ vanishes the $f_i$, then $hf(\gamma)=0$, and so $f(\gamma)=0$ since $h(\gamma)\neq0$.
\end{proof}

\begin{proof}[Proof of Lemma \ref{lem:CoV}]
We first notice that \ref{item:CoV1prime} is a consequence of \ref{item:CoV1}: $\forall\pi_n(\gamma)\in\Delta_{e,e',n}$ we have
\begin{align}
\begin{split}
\pi_n(\gamma)\in\sigma_{*n}^{-1}(\sigma_{*n}(\pi_n(\gamma)))
    &= \left\{\pi_n(\eta),\,\eta\in\mathcal L(M),\,\sigma(\eta)\equiv\sigma(\gamma)\mod t^{n+1}\right\}\text{ using that $\mathcal L(M)\rightarrow\mathcal L_n(M)$}\\
    &  \text{is surjective since $M$ is smooth and that $\pi_n\circ\sigma_*=\sigma_{*n}\circ\pi_n$.}
\end{split}\nonumber \\
    &\subset \left\{\eta\in\mathcal \Delta_{e,e',n},\,\gamma\equiv\eta\mod t^{n-e+1}\right\}\subset\Delta_{e,e',n} \text{ by \ref{item:CoV1}}\nonumber
\end{align}
Next \ref{item:CoV1} is a direct consequence of \ref{item:CoV0}. We apply \ref{item:CoV0} to $\gamma$ with $\delta=\sigma_*(\eta)$, hence there exists a unique $\tilde\eta$ such that $\tilde\eta\equiv\gamma\mod t^{n-e+1}$ and $\sigma_*(\tilde\eta)=\sigma_*(\eta)$. 
By the assumptions on $\sigma$ and the definition of $\Delta_{e,e'}$, for $\varphi_1\in\mathcal L(M)$ and $\varphi_2\in\Delta_{e,e'}$ with $\varphi_1\neq\varphi_2$ we have $\sigma(\varphi_1)\neq\sigma(\varphi_2)$. 
Hence $\eta=\tilde\eta$ and $\eta\equiv\gamma\mod t^{n-e+1}$. Since $\sigma(\gamma)\equiv\sigma(\eta)\mod t^{n+1}$ and $n\ge e'$, $\sigma(\eta)\in\mathcal L^{(e')}(X)$. We may write $\eta(t)=\gamma(t)+t^{n+1-e}u(t)$ and applying Taylor expansion to $\Jac_\sigma(\gamma(t)+t^{n+1-e}u(t))$ we get that $\Jac_\sigma(\eta(t))\equiv\Jac_\sigma(\gamma(t))\mod t^{e+1}$ since $n+1-e\ge e+1$. So $\eta\in\Delta_{e,e'}$.

So we just have to prove \ref{item:CoV0} and \ref{item:CoV2}.

We begin to refine the cover of Lemma \ref{lem:CI}: for $e''\le e'$, we set $$A_{h,\delta,e''}=\left\{\gamma\in A_{h,\delta},\,\ord_t\delta(\gamma)=e''\text{and $\ord_t\delta'(\gamma)\ge e''$ for all $(N-d)$-minor $\delta'$ of $\left(\frac{\partial f_i}{\partial x_j}\right)_{{i=1,\ldots,N-d\atop j=1,\ldots,N}}$}\right\}$$
Fix some $A=A_{h,\delta,e''}$, then it suffices to prove the lemma for $\Delta_{e,e'}\cap\sigma^{-1}(A)$. \\
Up to renumbering the coordinates, we may also assume that $\delta$ is the determinant of the first $N-d$ columns of $\Delta=\left(\frac{\partial f_i}{\partial x_j}\right)_{{i=1,\ldots,N-d\atop j=1,\ldots,N}}$.

We choose a local coordinate system of $M$ at $\gamma(0)$ in order to define $\Jac_\sigma$ and express arcs of $M$ as elements of $\mathbb R\{t\}^d$.

Now, a crucial observation is that the first $N-d$ rows of $\Jac_\sigma(\gamma)$ are $\mathbb R\{t\}$-linear combinations of the last $d$ rows: 
the application $$\begin{array}{ccccc}M&\longrightarrow&X&\longrightarrow&\mathbb R^{N-d}\\y&\longmapsto&\sigma(y)&\longmapsto&\left(f_i(\sigma(y))\right)_{i=1,\ldots,N-d}\end{array}$$ is identically zero, so its Jacobian matrix is identically zero too and thus $\Delta(\sigma(\gamma))\Jac_\sigma(\gamma)=0$. Let $P$ be the transpose of the comatrix of the submatrix of $\Delta$ given by the first $N-d$ columns of $\Delta$, then $P\Delta=(\delta I_{N-d},W)$. Moreover, we have $W(\sigma(\gamma))\equiv0\mod t^{e''}$. Indeed, if we denote $\Delta_1,\ldots,\Delta_{N-d}$ the $N-d$ first columns of $\Delta$ and $W_1,\ldots,W_d$ the columns of $W$, then $W_j(\sigma(\gamma))$ is solution of $\left(\Delta_1(\sigma(\gamma)),\ldots,\Delta_{N-d}(\sigma(\gamma))\right)X=\delta(\sigma(\gamma))\Delta_{N-d+j}(\sigma(\gamma))$ since
\begin{align*}\delta(\sigma(\gamma))\Delta(\sigma(\gamma))&=\left(\Delta_1(\sigma(\gamma)),\ldots,\Delta_{N-d}(\sigma(\gamma))\right)P(\sigma(\gamma))\Delta(\sigma(\gamma))\\&=\left(\Delta_1(\sigma(\gamma)),\ldots,\Delta_{N-d}(\sigma(\gamma))\right)\left(\delta(\sigma(\gamma))I_{N-d},W(\sigma(\gamma))\right)
\end{align*}
So, by Cramer's rule, $$\left(W_j(\sigma(\gamma))\right)_i=\det\left(\Delta_1(\sigma(\gamma)),\ldots,\Delta_{i-1}(\sigma(\gamma)),\Delta_{N-d+j}(\sigma(\gamma)),\Delta_{i+1}(\sigma(\gamma)),\ldots,\Delta_{N-d}(\sigma(\gamma))\right)$$
Finally, the congruence arises because the minor formed by the $N-d$ first columns is of minimal order by definition of $A$. \\
Now the columns of $\Jac_\sigma(\gamma)$ are solutions of
\begin{equation}
\left(t^{-e''}\cdot P(\sigma(\gamma)))\cdot\Delta(\sigma(\gamma))\right)X=0\label{eq:jac}
\end{equation}
 but since $t^{-e''}\cdot P(\sigma(\gamma)))\cdot\Delta(\sigma(\gamma))=\left(t^{-e''}\delta(\sigma(\gamma))I_{N-d},t^{-e''}W(\sigma(\gamma))\right)$ we may express the first $N-d$ coordinates of each solution in terms of the last $d$ coordinates. This completes the proof of the observation.

For \ref{item:CoV0}, it suffices to prove that for all $v\in\mathbb R\{t\}^N$ satisfying $\sigma(\gamma)+t^{n+1}v\in\mathcal L(X)$ there exists a unique $u\in\mathbb R\{t\}^d$ such that \begin{equation}\sigma(\gamma+t^{n+1-e}u)=\sigma(\gamma)+t^{n+1}v\label{eq:CoV1}\end{equation}
By Taylor expansion, we have \begin{equation}\sigma(\gamma(t)+t^{n+1-e}u)=\sigma(\gamma(t))+t^{n+1-e}\Jac_\sigma(\gamma(t))u+t^{2(n+1-e)}R(\gamma(t),u)\label{eq:CoV2}\end{equation} with $R(\gamma(t),u)$ analytic in $t$ and $u$. By \eqref{eq:CoV2}, \eqref{eq:CoV1} is equivalent to \begin{equation}t^{-e}\Jac_\sigma(\gamma(t))u+t^{n+1-2e}R(\gamma(t),u)=v\label{eq:CoV3}\end{equation} with $n+1-2e\ge1$ by hypothesis. \\
Since $\sigma(\gamma(t))+t^{n+1}v\in\mathcal L(X)$ and using Taylor expansion, we get $$0=f_i(\sigma(\gamma(t))+t^{n+1}v)=t^{n+1}\Delta(\sigma(\gamma(t)))v+t^{2(n+1)}S(\gamma(t),v)$$
with $S(\gamma(t),v)$ analytic in $t$ and $v$. So $v$ is a solution of \eqref{eq:jac} and hence the first $N-d$ coefficients of $v$ are $\mathbb R\{t\}$-linear combinations of the last $d$ coefficients with the same relations that for $\Jac_\sigma(\gamma)$. This allows us to reduce \eqref{eq:CoV3} to \begin{equation}t^{-e}\Jac_{p\circ\sigma}(\gamma(t))u+t^{n+1-2e}p\left(R(\gamma(t),u)\right)=p(v)\label{eq:CoV4}\end{equation} where $p:\mathbb R^N\rightarrow\mathbb R^{d}$ is the projection on the last $d$ coordinates. The observation ensures that $\ord_t\Jac_{p\circ\sigma}(\gamma(t))=\ord_t\Jac_{\sigma}(\gamma(t))=e$ and thus \eqref{eq:CoV4} is equivalent to \begin{equation}u=\left(t^{-e}\Jac_{p\circ\sigma}(\gamma(t))\right)^{-1}p(v)-t^{n+1-2e}\left(t^{-e}\Jac_{p\circ\sigma}(\gamma(t)\right)^{-1}p\left(R(\gamma(t),u)\right)\end{equation}
Applying the implicit function theorem to $u(t,v)$ ensures that given an analytic arc $v(t)$ there exists a solution $u_v(t)=u(t,v(t))$. Using the same argument as in the proof of \ref{item:CoV1}, the solution $u_v(t)$ is unique. 
This proves \ref{item:CoV0}.

Let us prove \ref{item:CoV2}. Let $\gamma\in\Delta_{e,e'}\cap\sigma^{-1}(A)$ then
\begin{align}
\sigma_{*n}^{-1}(\pi_n(\sigma_*(\gamma)))
    &= \left\{\eta\in\mathcal L_n(M),\,\sigma_{*n}(\eta)=\pi_n(\sigma_*(\gamma)\right\}\label{eq:fibre} \nonumber\\
\begin{split}
    &= \left\{\pi_n(\eta),\,\eta\in\mathcal L(M),\,\sigma(\eta)\equiv\sigma(\gamma)\mod t^{n+1}\right\}\text{ using that $\mathcal L(M)\rightarrow\mathcal L_n(M)$ is}\\
    &  \text{surjective since $M$ is smooth and that $\pi_n\circ\sigma_*=\sigma_{*n}\circ\pi_n$.}
\end{split}\nonumber \\
    &= \left\{\gamma(t)+t^{n+1-e}u(t)\mod t^{n+1},\,u\in\mathbb R\{t\}^d,\,\Jac_{p\circ\sigma}(\gamma(t))u(t)\equiv0\mod t^e\right\}\nonumber \\
    & \text{ by \ref{item:CoV1} and \eqref{eq:CoV4}}\nonumber
\end{align}
Thus, the fiber is an affine subspace of $\mathbb R^{de}$. There are invertible matrices $A$ and $B$ with coordinates in $\mathbb R\{t\}$ such that $A\Jac_{p\circ\sigma}(\gamma(t))B$ is diagonal with entries $t^{e_1},\ldots,t^{e_d}$ such that $e=e_1+\cdots+e_d$. Therefore the fiber is of dimension $e$.

Since $\sigma$ is not assumed to be birational, we can't use the section argument of \cite[3.4]{DL99} or \cite[4.2]{KP03}, instead we use a topological noetherianity argument to prove that $\sigma_{*n|\Delta_{e,e',n}}$ is a piecewise trivial fibration.

We may assume that $M$ is semialgebraically connected, then by Artin-Mazur theorem \cite[8.4.4]{BCR}, there exist $Y\subset\mathbb R^{p+q}$ a non-singular irreducible algebraic set of dimension $\dim M$, $M'\subset Y$ an open semialgebraic subset of $Y$, $s:M\rightarrow M'$ a Nash-diffeomorphism and $g:Y\rightarrow\mathbb R^N$ a polynomial map such that the following diagram commutes
$$\xymatrix{
\mathbb R^{p+q} \ar[dd]_\Pi & Y \ar@{_{(}->}[l] \ar[rd]^g & \\
& M' \ar@{^{(}->}[u] & \mathbb R^N \\
\mathbb R^p & M \ar[u]^s_\simeq \ar[ru]_\sigma \ar@{_{(}->}[l]   &
}$$
Thus, we have
$$\sigma_{*n}^{-1}(\pi_n(\sigma_*(\gamma)))=\left\{\gamma(t)+t^{n+1-e}u(t)\mod t^{n+1},\,u\in\mathbb R\{t\}^d,\,\Jac_{g\circ s}(\gamma(t))u(t)\equiv0\mod t^e\right\}$$
So $\Delta_{e,e',n}$ is constructible and we may assume that $\sigma_{*n}:\Delta_{e,e',n}\rightarrow\sigma_{*n}(\Delta_{e,e',n})$ is polynomial up to working with arcs over $M'$ via $s$. The fibers (i.e. $\mathbb R^e$) have odd Euler characteristic with compact support, so by Theorem \ref{thm:oddfibers} the image $\sigma_{*n}(\Delta_{e,e',n})$ is constructible.

Let $V=\{u_0+u_1t+\cdots+u_nt^n,\,u_i\in\mathbb R^d\}$ and fix $\Lambda_0:V\rightarrow V_0$ a linear projection on a subspace of dimension $e$. The set $\Omega_{0}=\{\pi_n(\gamma(t))\in\Delta_{e,e',n},\,\dim\Lambda_0(\sigma_{*n}^{-1}(\pi_n(\sigma_*(\gamma))))<e\}$ is closed, constructible and union of fibers of $\sigma_{*n}$. Therefore $(\sigma_{*n},\Lambda_0):\Delta_{e,e',n}\setminus\Omega_{0}\rightarrow\sigma_{*n}(\Delta_{e,e',n}\setminus\Omega_{0})\times V_0$ is a constructible isomorphism. We now repeat the argument to the closed constructible subset $\sigma_{*n}(\Omega_{0})$ and so on. Indeed, assume that $\Delta_{e,e',n}\supsetneq\Omega_{0}\supsetneq\Omega_{1}\supsetneq\cdots\supsetneq\Omega_{{i-1}}$ are constructed as previously and that $\Omega_{{i-1}}\neq\varnothing$, then we may choose $\Lambda_i$ such that $\Omega_{i}\subsetneq\Omega_{{i-1}}$. So on the one hand the process continues until one $\Omega_i$ is empty, on the other hand it must stop because of the noetherianity of the $\mathcal{AS}$-topology. Therefore after a finite number of steps, one $\Omega_i$ is necessarily empty.
\end{proof}

\subsection{Essence of the proof}
By our hypothesis, there exists a sequence of blowings-up $\sigma:M\rightarrow X$ with non-singular centers such that $\tilde\sigma=f\circ\sigma:M\rightarrow X$ is Nash.
$$\xymatrix{&M \ar[ld]_\sigma \ar[rd]^{\tilde\sigma}&\\X \ar[rr]_f&&X}$$
After adding more blowings-up, we may assume that the critical loci of $\sigma$ and $\tilde\sigma$ are simultaneously normal crossing and denote them by $\sum\nu_iE_i$ and $\sum\tilde\nu_iE_i$. Our hypothesis ensures that $\nu_i\ge\tilde\nu_i$.

In the same way, we may ensure that the inverse images of $H$ (defined in Lemma \ref{lem:Hideal}) by $\sigma$ and $\tilde\sigma$ are also simultaneously normal crossing and denote them $\sigma^{-1}(H)=\sum_{i\in I}\lambda_iE_i$ (resp. $\tilde\sigma^{-1}(H)=\sum_{i\in I}\tilde\lambda_iE_i$). \\

We recall the usual notation\footnote{This notation is natural and classical. See \cite[Chapter II, \S1]{KKMS} for some properties of this stratification.}. For $\mathbf{j}=(j_i)_{i\in I}\in\mathbb N^I$, we set $J=J(\mathbf j)=\{i,\,j_i\neq0\}\subset I$, $E_J=\cap_{i\in J}E_i$ and $\pring{E_J}=E_J\setminus\cup_{i\in I\setminus J}E_i$.

We also define: $\mathcal B_\mathbf j=\{\gamma\in\mathcal L(M),\,\forall i\in J,\,\ord_\gamma E_i=j_i,\,\gamma(0)\in\pring{E_J}\}$ and for all $n\in\mathbb N$, $\mathcal B_{\mathbf{j},n}=\pi_n(\mathcal B_\mathbf j)$ and $X_{\mathbf j,n}(\sigma)=\pi_n(\sigma_*\mathcal B_\mathbf j)=\sigma_{*n}(\mathcal B_{\mathbf j,n})$. \\

\begin{lemma}\label{lem:eprime}
We have $\mathcal B_{\mathbf j}\subset\Delta_{e(\mathbf j),e'(\mathbf j)}(\sigma)$ where $\displaystyle e(\mathbf j)=\sum_{i\in I}\nu_ij_i$ and $\displaystyle e'(\mathbf j)=\sum_{i\in I}\lambda_ij_i$.
\end{lemma}
\begin{proof}
Let $\gamma\in\mathcal B_{\mathbf j}$ and choose a local coordinate system of $M$ at $\gamma(0)$ such that the critical locus of $\sigma$ is locally described by the equation $\prod_{i\in J}x_i^{\nu_i}=0$ and $E_i$ by the equation $x_i=0$. Since $\ord_\gamma E_i=j_i$, we have $\gamma_i(t)=c_{j_i}t^{j_i}+\cdots$ and $c_{j_i}\neq0$. Then $\prod_{i\in J}\gamma_i^{\nu_i}=ct^{e(\mathbf j)}+\cdots$ with $c\neq0$. \\
So we have $\ord_t\left(\Jac_\sigma(\gamma(t))\right)=e(\mathbf j)$. \\
In the same way, $\ord_{\gamma}\sigma^{-1}(H)=e'(\mathbf j)$ thus $\ord_{\sigma(\gamma)}(H)=e'(\mathbf j)$.
\end{proof}

Therefore we set $A_n(\sigma)=\left\{\mathbf j,\,\sum_{i\in I}\nu_ij_i\le\frac{n}{2},\,\sum_{i\in I}\lambda_ij_i\le n\right\}$. Indeed, for each $\mathbf j\in A_n(\sigma)$, $\mathcal B_{\mathbf j}\subset\Delta_{e(j),e'(j)}(\sigma)$ and we may apply Lemma \ref{lem:CoV} at the level of $n$-jets. \\

The argument of the following lemma is essentially the same as \cite[\S4.2]{FKP10}.
\begin{lemma}[A decomposition of jet spaces]\label{lem:Decomposition}
For all $\mathbf j\in A_n(\sigma)$, the sets $X_{\mathbf j,n}(\sigma)$ are constructible subsets of $\mathcal L_n(X)$ and $\dim X_{\mathbf j,n}(\sigma)=d(n+1)- s_\mathbf j-\sum_{i\in I}\nu_ij_i$ where $s_\mathbf j=\sum_{i\in I} j_i$. Moreover $\Im(\sigma_{*n})=Z_n(\sigma)\sqcup\displaystyle\bigsqcup_{\mathbf j\in A_n(\sigma)}X_{\mathbf j,n}(\sigma)$ and the set $Z_n(\sigma)$ satisfies $\dim Z_n(\sigma)<d(n+1)-\frac{n}{c}$ where $c=\max(2\nu_{\max},\lambda_{\max})$.
\end{lemma}
\begin{proof}
Consider $\mathbf j$ such that $\pring{E_J}\neq\varnothing$ and $\forall i\in I,\,0\le j_i\le n$. The fiber of $\mathcal B_{\mathbf j,n}\rightarrow\pring{E_J}$ is $$\prod_{i\in J}(\mathbb R^*\times\mathbb R^{n-j_i})\times(\mathbb R^n)^{d-|J|}\simeq(\mathbb R^*)^{|J|}\times\mathbb R^{dn-s_\mathbf j}$$ since truncating the coordinates of $\gamma\in\mathcal B_\mathbf j$ to degree $n$ produces $d-|J|$ polynomials of degree $n$ with fixed constant terms and for $i\in J$ a polynomial of the form $c_{j_i}t^{j_i}+c_{j_i+1}t^{j_i+1}+\cdots+c_{n}t^{n}$ with $c_{j_i}\in\mathbb R^*$ and other $c_k\in\mathbb R$. We conclude that $\dim \mathcal B_{\mathbf j,n}=d(n+1)-s_\mathbf j$.

We first assume that $\mathbf j\in A_n(\sigma)$. By Lemma \ref{lem:eprime}, $\mathcal B_{\mathbf j}\subset\Delta_{e(\mathbf j),e'(\mathbf j)}(\sigma)$. Hence by \ref{item:CoV2}, $X_{\mathbf j,n}(\sigma)$ is constructible since it is the image of the constructible set $\mathcal B_{\mathbf j,n}$ by the map $\sigma_{*n|\Delta_{e(\mathbf j),e'(\mathbf j),n}}$ with fibers of odd Euler characteristic with compact support. Let $\gamma_1\in\mathcal B_{\mathbf j,n}$ and $\gamma_2\in\Delta_{e(\mathbf j),e'(\mathbf j),n}$ with $\sigma_{*n}(\gamma_1)=\sigma_{*n}(\gamma_2)$, then, by \ref{item:CoV1}, $\gamma_1\equiv\gamma_2\mod t^{n-e(\mathbf j)+1}$ with $n-e(\mathbf j)\ge e(\mathbf j)$ and hence $\gamma_2\in\mathcal B_{\mathbf j,n}$. Thus by \ref{item:CoV2} the map $\mathcal B_{\mathbf j,n}\rightarrow X_{\mathbf j,n}(\sigma)$ is a piecewise trivial fibration with fiber $\mathbb R^{e(\mathbf j)}$. So we have $\dim X_{\mathbf j,n}(\sigma)=d(n+1)-s_\mathbf j-e(\mathbf j)$ as claimed.

Otherwise $\mathbf j\notin A_n(\sigma)$ and then $\dim X_{\mathbf j,n}\le\dim\mathcal B_{\mathbf j,n}=d(n+1)-s_\mathbf j< d(n+1)-\frac{n}{c}$ (since $\frac{n}{2}<e(\mathbf j)\le\nu_{\max} s_\mathbf j$ or $n<e'(\mathbf j)\le\lambda_{\max}s_\mathbf j$).
\end{proof}

\begin{rem}
The two previous lemmas work as they are if we replace $\sigma$ by $\tilde\sigma$, $\nu_i$ by $\tilde\nu_i$, $\lambda_i$ by $\tilde\lambda_i$ and $c$ by $\tilde c$.
\end{rem}

\begin{rem}
Remember that $\Im\sigma_{*n}\subset\pi_n(\mathcal L(X))$ (resp. $\Im\tilde\sigma_{*n}\subset\pi_n(\mathcal L(X))$). Moreover, since we may lift by $\sigma$ an arc not entirely included in the singular locus, $\pi_n(\mathcal L(X))\setminus\Im\sigma_{*n}\subset\pi_n(\mathcal L(X_\sing))$. The second part only works for $\sigma$ and doesn't stand for $\tilde\sigma$.
\end{rem}

In order to apply the virtual Poincaré polynomial, we are going to modify the objects of the partitions of Lemma \ref{lem:Decomposition}.
\begin{notation}\label{notation:partitions}
We set $$\widetilde{\pi_n(\mathcal L(X))}:=\clos[\mathcal AS]{Z_n(\sigma)\sqcup(\pi_n(\mathcal L(X))\setminus\Im\sigma_{*n})}\sqcup\bigsqcup_{\mathbf j\in A_n(\sigma)}X_{\mathbf j,n}(\sigma)$$ $$\left(\text{resp. }\widetilde{\Im\tilde\sigma_{*n}}:=\clos[\mathcal AS]{Z_n(\tilde\sigma)}\sqcup\bigsqcup_{\mathbf j\in A_n(\tilde\sigma)}X_{\mathbf j,n}(\tilde\sigma)\right)$$ where the closure is taken in the complement of $\displaystyle\bigsqcup_{\mathbf j\in A_n(\sigma)}X_{\mathbf j,n}(\sigma)\ \displaystyle\left(\text{resp. in }\widetilde{\pi_n(\mathcal L(X))}\setminus\bigsqcup_{\mathbf j\in A_n(\tilde\sigma)}X_{\mathbf j,n}(\tilde\sigma)\right)$. Hence we still have the inclusion $\widetilde{\Im\tilde\sigma_{*n}}\subset\widetilde{\pi_n(\mathcal L(X))}$, the unions are still disjoint and the dimensions remain the same.
\end{notation}

\begin{lemma}
For $\mathbf j\in A_n(\sigma)$ we have $\beta\left(X_{\mathbf j,n}(\sigma)\right)=\beta\left(\pring{E_J}\right)(u-1)^{|J|}u^{nd-\sum(\nu_i+1)j_i}$. \\ (resp. for $\mathbf j\in A_n(\tilde\sigma)$ we have $\beta\left(X_{\mathbf j,n}(\tilde\sigma)\right)=\beta\left(\pring{E_J}\right)(u-1)^{|J|}u^{nd-\sum(\tilde\nu_i+1)j_i}$)
\end{lemma}
\begin{proof}
We have
\begin{align*}
\beta\left(X_{\mathbf j,n}(\sigma)\right)
&=\beta\left(\mathcal B_{\mathbf j,n}\right)u^{-\sum\nu_ij_i}\ \ \text{ by Lemma \ref{lem:CoV} and Lemma \ref{lem:eprime}} \\
&=\beta\left(\pring{E_J}\times(\mathbb R^*)^{|J|}\times\mathbb R^{dn-s_\mathbf j}\right)u^{-\sum\nu_ij_i}\ \ \text{ by the beginning of the proof of Lemma \ref{lem:Decomposition}} \\
&=\beta\left(\pring{E_J}\right)(u-1)^{|J|}u^{nd-s_\mathbf j-\sum\nu_ij_i}
\end{align*}
The same argument works for $\tilde\sigma$ too.
\end{proof}

\begin{lemma}\label{lem:nuinuitilde}
$\forall i\in I,\,\nu_i=\tilde\nu_i$
\end{lemma}

\begin{proof}
Applying the virtual Poincaré polynomial to the partitions of Notation \ref{notation:partitions}, we get
\begin{multline*}
\beta\left(\widetilde{\pi_n(\mathcal L(X))}\right)-\beta\left(\widetilde{\Im\tilde\sigma_{*n}}\right)-\sum_{\mathbf j\in A_n(\sigma)\cap A_n(\tilde\sigma)}\left(\beta(X_{\mathbf j,n}(\sigma))-\beta(X_{\mathbf j,n}(\tilde\sigma))\right) \\
=\sum_{\mathbf j\in A_n(\sigma)\setminus A_n(\tilde\sigma)}\beta(X_{\mathbf j,n}(\sigma))-\sum_{\mathbf j\in A_n(\tilde\sigma)\setminus A_n(\sigma)}\beta(X_{\mathbf j,n}(\tilde\sigma))+\beta\left(\clos[\mathcal{AS}]{Z_n(\sigma)\sqcup(\pi_n(\mathcal L(X))\setminus\Im\sigma_{*n})}\right)-\beta\left(\clos[\mathcal{AS}]{Z_n(\tilde\sigma)}\right)
\end{multline*}
We set
$$
\begin{array}{ll}
\displaystyle P_n=\beta\left(\widetilde{\pi_n(\mathcal L(X)}\right)-\beta\left(\widetilde{\Im\tilde\sigma_{*n}}\right), & \displaystyle Q_n=-\sum_{\mathbf j\in A_n(\sigma)\cap A_n(\tilde\sigma)}\left(\beta(X_{\mathbf j,n}(\sigma))-\beta(X_{\mathbf j,n}(\tilde\sigma))\right), \\
\displaystyle R_n=\sum_{\mathbf j\in A_n(\sigma)\setminus A_n(\tilde\sigma)}\beta(X_{\mathbf j,n}(\sigma)), & \displaystyle S_n=-\sum_{\mathbf j\in A_n(\tilde\sigma)\setminus A_n(\sigma)}\beta(X_{\mathbf j,n}(\tilde\sigma)), \\
\displaystyle T_n=\beta\left(\clos[\mathcal{AS}]{Z_n(\sigma)\sqcup(\pi_n(\mathcal L(X))\setminus\Im\sigma_{*n})}\right), & \displaystyle U_n=-\beta\left(\clos[\mathcal{AS}]{Z_n(\tilde\sigma)}\right).
\end{array}$$
Assume there exists $i_0\in I$ such that $\nu_{i_0}>\tilde\nu_{i_0}$.

Then for $n$ big enough, $K_n=\displaystyle\left\{s_\mathbf j+\sum_{i\in I}\tilde\nu_ij_i,\ \mathbf j\in A_n(\sigma)\cap A_n(\tilde\sigma),\,\sum_{i\in I}(\nu_i-\tilde\nu_i)j_i>0\right\}$ is not empty. The minimum $k_n=\min K_n$ stabilizes for $n$ greater than some rank $n_0$. Let $k=k_{n_0}$. Then, for $n\ge n_0$, the degree of $Q_n$ is $\max\left\{d(n+1)-s_\mathbf j-\sum_{i\in I}\tilde\nu_ij_i\right\}=d(n+1)-k$ using the computation at the beginning of the proof of Lemma \ref{lem:Decomposition}.

The leading coefficients of $P_n$ is positive since $P_n=\beta\left(\widetilde{\pi_n(\mathcal L(X))}\setminus\widetilde{\Im\tilde\sigma_{*n}}\right)$. The leading coefficient of $Q_n$ is also positive. Hence the degree of the LHS is at least $d(n+1)-k$. \\
Moreover, we have $\deg R_n<d(n+1)-\frac{n}{\tilde c}$, $\deg S_n<d(n+1)-\frac{n}{c}$, $\deg T_n<d(n+1)-\frac{n}{\max(c,1)}$ and $\deg U_n<d(n+1)-\frac{n}{\tilde c}$. Indeed, for $T_n$, $\pi_n(\mathcal L(X))\setminus\Im\sigma_{*n}\subset\pi_n(\mathcal L(X_\sing))$ and $\dim\left(\pi_n(\mathcal L(X_\sing))\right)\le (n+1)(d-1)<d(n+1)-n$ by \ref{prop:jetsdim}.\ref{item:dimliftable}. So the degree of the RHS is less than $d(n+1)-\frac{n}{\max(c,\tilde c,1)}$. \\
We get a contradiction for $n$ big enough.
\end{proof}

\begin{cor}
$Q_n=0$
\end{cor}

Since $\tilde\sigma:M\rightarrow X$ is a proper Nash map generically one-to-one, there exists a closed semialgebraic subsets $S\subset X$ with $\dim S<d$ such that for every $p\in X\setminus S$, $\tilde\sigma^{-1}(p)$ is a singleton.

\begin{cor}\label{cor:lift}
Every arc on $X$ not entirely included in $S\cup X_\sing$ may be uniquely lifted by $\tilde\sigma$.
\end{cor}
\begin{proof}
Let $\gamma$ be an analytic arc on $X$ not entirely in $S$ and not entirely in the singular locus of $X$. \\
Assume that $\gamma\notin\Im\tilde\sigma_*$. Then, by Proposition \ref{prop:Puiseux}, we have
$$\tilde\sigma^{-1}(\gamma(t))=\sum_{i=0}^mb_it^i+bt^{\frac{p}{q}}+\cdots,\,b\neq 0,\,m<\frac{p}{q}<m+1,\,t\ge0$$
Since $\tilde\sigma^{-1}$ is locally Hölder by Remark \ref{rem:Holder}, there is $N\in\mathbb N$ such that for every analytic arc $\eta$ on $X$ with $\gamma\equiv\eta\mod t^{N}$ we have $\tilde\sigma^{-1}(\eta(t))\equiv\tilde\sigma^{-1}(\gamma(t))\mod t^{m+1}$. Hence such an analytic arc $\eta$ isn't in the image of $\tilde\sigma_*$ and for $n\ge N$, $\pi_n(\eta)$ isn't in the image of $\tilde\sigma_{*n}:\mathcal L_n(M)\rightarrow\mathcal \pi_n(\mathcal L(X))$. Hence $\left({\pi^n_N}_{|\pi_n(\mathcal L(X))}\right)^{-1}(\pi_N(\gamma))\subset\pi_n(\mathcal L(X))\setminus\Im(\tilde\sigma_{*n})$.

The first step consists in computing the dimension of the fiber $\left({\pi^n_N}_{|\pi_n(\mathcal L(X))}\right)^{-1}(\pi_N(\gamma))$ where $n\ge N$. For that, we will work with a resolution $\rho:\tilde X\rightarrow X$ (for instance $\sigma$) instead of $\tilde\sigma$ since every analytic arc on $X$ not entirely included in $X_\sing$ may be lifted to $\tilde X$ by $\rho$. Let $\theta$ be the unique analytic arc on $\tilde X$ such that $\rho(\theta)=\gamma$. Let $e=\ord_t\left(\Jac_\rho(\theta(t))\right)$ and $e'$ be such that $\gamma\in\mathcal L^{(e')}(X)$. We may assume that $N\ge\max(2e,e')$ in order to apply Lemma \ref{lem:CoV} to $\rho$ for $\gamma$. \\
We consider the following diagram
$$\xymatrix{
\mathcal L(\tilde X)  \ar[r]^{\rho_*} \ar@{->>}[d]_{\pi_n} & \mathcal L(X) \ar@{->>}[d]^{\pi_n} \\%& \mathcal L(M) \ar@{.>}[l]_{\tilde\sigma_*} \ar@{.>>}[d]^{\pi_n}\\
\mathcal L_n(\tilde X) \ar[r]^{\rho_{*n}} \ar@{->>}[d]_{\pi^n_N} & \pi_n(\mathcal L(X)) \ar@{->>}[d]^{\pi^n_N} \\%& \mathcal L_n(M)  \ar@{.>}[l]_{\tilde\sigma_{*n}} \ar@{..>>}[d]^{\pi^n_N} \\
\mathcal L_N(\tilde X) \ar[r]_{\rho_{*N}} & \pi_N(\mathcal L(X)) 
}$$
Since the fibers of ${\rho_{*n}}_{|\Delta_{e,e',n}}$ and ${\rho_{*N}}_{|\Delta_{e,e',N}}$ are of dimension $e$, and since the fibers of $\pi^n_N:\mathcal L_n(\tilde X)\rightarrow\mathcal L_N(\tilde X)$ are of dimension $(n-N)d$, we have $\dim\left(\left({\pi^n_N}_{|\pi_n(\mathcal L(X))}\right)^{-1}(\pi_N(\gamma))\right)=(n-N)d$. \\
Hence $\dim\left(\pi_n(\mathcal L(X))\setminus\Im(\tilde\sigma_{*n})\right)\ge (n-N)d$. And so, with the notation of Lemma \ref{lem:nuinuitilde}, we have $$P_n+0=R_n+S_n+T_n+U_n$$ with $\deg P_n\ge(n-N)d=(n+1)d-(N+1)d$ and $\deg(R_n+S_n+T_n+U_n)<(n+1)d-\frac{n}{\max(c,\tilde c,1)}$. \\
We get a contradiction for $n$ big enough.
\end{proof}

\begin{proof}[End of the proof of Theorem \ref{thm:Main}]
Let $\gamma$ be an analytic arc on $X$ not entirely included in $S\cup X_\sing$. By Corollary \ref{cor:lift} and since $\gamma$ is not entirely included in $S\cup X_\sing$, $\tilde\sigma^{-1}(\gamma(t))$ is well defined and analytic. Hence $f^{-1}(\gamma(t))=\sigma(\tilde\sigma^{-1}(\gamma(t)))$ is real analytic. Finally $f^{-1}$ is generically arc-analytic in dimension $d=\dim X$. \\
So $f^{-1}$ is blow-Nash by Proposition \ref{prop:blowNash} and $\forall i\in I,\,\nu_i=\tilde\nu_i$ by Lemma \ref{lem:nuinuitilde}. Then, arguing as in Lemma \ref{lem:hyp}, $f^{-1}$ satisfies the Jacobian hypothesis too.
\end{proof}

\nocite{Hir64} 
\nocite{Fic05bis} 
\nocite{KP05} 
\nocite{FKP04}
\nocite{Seb04}
\nocite{BMP91}
\nocite{Paw84}
\footnotesize
%\bibliography{IFT}

\end{document}